\documentclass[preprint, 10pt]{elsarticle}

\usepackage{amsmath}
\usepackage{amssymb}
\usepackage{amsthm}
\usepackage{enumitem}
\usepackage{graphicx}
\usepackage{url}
\usepackage{color}

\theoremstyle{plain}
\newtheorem{Thm}{Theorem}[section]
\newtheorem{Cor}[Thm]{Corollary}
\newtheorem{Lem}[Thm]{Lemma}
\newtheorem{Prop}[Thm]{Proposition}

\theoremstyle{definition}
\newtheorem{Def}[Thm]{Definition}
\newtheorem{Rem}[Thm]{Remark}

\theoremstyle{remark}

\numberwithin{Thm}{section}
\numberwithin{equation}{section}

\newcommand{\N}{\mathbb{N}}

\newcommand{\R}{\mathbb{R}}

\newcommand{\cl}{\overline}
\newcommand{\del}{\partial}
\newcommand{\loc}{\mathrm{loc}}

\DeclareMathOperator{\divergence}{div}

\newcommand{\laplacian}{\Delta}
\DeclareMathOperator*{\D}{D\!}

\DeclareMathOperator*{\spt}{supp}
\newcommand{\capacity}{\mathrm{cap}}

\newcommand{\M}{\mathcal{M}}

\newcommand{\W}{{\bf{W}}}
\newcommand{\I}{{\bf{I}}}

\newcommand{\XXint}[3]
{{
    \setbox0=\hbox{$#1{#2#3}{\int}$}
    \vcenter{\hbox{$#2#3$}}\kern-.5\wd0
}}

\newcommand{\trinorm}[1]
{{
    \left\vert\kern-0.20ex\left\vert\kern-0.20ex\left\vert
    #1 
    \right\vert\kern-0.20ex\right\vert\kern-0.20ex\right\vert
}}

\begin{document}
\begin{frontmatter}




\title{Existence of minimal solutions to quasilinear elliptic equations with several sub-natural growth terms \tnoteref{t1}}

\author{Takanobu Hara\corref{cor1}\fnref{fn1}}
\ead{takanobu.hara.math@gmail.com}

\author{Adisak Seesanea\fnref{fn1}}
\ead{adisak.seesanea@gmail.com}

\address{Department of Mathematics, Hokkaido University,
                Kita 8 Nishi 10  Sapporo, Hokkaido 060-0810, Japan}

\cortext[cor1]{Corresponding author}

\tnotetext[t1]{
This is a pre-print of an article published in \textit{Nonlinear Analysis}.
The final authenticated version is available online at: \url{https://doi.org/10.1016/j.na.2020.111847}.
}


\begin{abstract}
We study the existence of positive solutions to
quasilinear elliptic equations of the type
\[ 
-\Delta_{p} u  = \sigma u^{q} + \mu \quad \text{in} \ \mathbb{R}^{n},
\] 
in the sub-natural growth case $0 < q < p - 1$, 
where $\Delta_{p}u = \nabla \cdot ( |\nabla u|^{p - 2} \nabla u )$
is the $p$-Laplacian with $1 < p < n$,
and $\sigma$ and $\mu$ are nonnegative Radon measures on $\mathbb{R}^{n}$.
We construct minimal generalized solutions
under certain generalized energy conditions on $\sigma$ and $\mu$.
To prove this, we give new estimates for interaction between measures.
We also construct solutions to equations with several sub-natural growth terms
using the same methods.
\end{abstract}


\end{frontmatter}




\section{Introduction and main results}

In this paper, we consider the model quasilinear elliptic problem
\begin{equation}\label{eq:p-laplacian}
\begin{cases}
\displaystyle
- \laplacian_{p} u = \sigma u^{q} + \mu, \quad u > 0 \quad \text{in} \ \R^{n}, \\
\displaystyle
\liminf_{|x| \to \infty} u(x) = 0,
\end{cases}
\end{equation}
in the sub-natural growth case $0<q< p-1$,
where $\laplacian_{p} u = \divergence( |\nabla u|^{p-2} \nabla u )$ is
the $p$-Laplacian with $1 < p < n$,
and $\sigma$ and $\mu$ are nonnegative Radon measures on $\R^{n}$.
We construct minimal generalized solutions to \eqref{eq:p-laplacian}
under certain generalized energy conditions on $\sigma$ and $\mu$.

When $\mu = 0$, Eq. \eqref{eq:p-laplacian} becomes
\begin{equation}\label{eq:p-laplacian_homo}
\begin{cases}
\displaystyle
- \laplacian_{p} u = \sigma u^{q}, \quad u > 0 \quad \text{in} \ \R^{n}, \\
\displaystyle
\liminf_{|x| \to \infty} u(x) = 0.
\end{cases}
\end{equation}
This equation is related to the trace inequality
\begin{equation}\label{intro:trace_ineq}
\| u \|_{L^{1 + q}(\R^{n}, d \sigma)}
\leq
C \| \nabla u \|_{L^{p}(\R^{n})}
\quad 
\forall u \in C_{c}^{\infty}(\R^{n}),
\end{equation}
where $\| \cdot \|_{L^{1 + q}(\R^{n}, d \sigma)}$ is the $L^{1 + q}$ norm with respect to
the measure $\sigma$.
Cascante, Ortega and Verbitsky \cite{MR1734322}
and Verbitsky  \cite{MR1747901}
proved that
\eqref{intro:trace_ineq} holds
if and only if  
\begin{equation}\label{intro:cond_sigma_1}
\int_{\R^{n}}
\left( \W_{1, p} \sigma \right)^{\frac{(1 + q)(p - 1)}{p - 1 - q}}
\, d \sigma < \infty,
\end{equation}
where
$\W_{1, p} \sigma$ is the Wolff potential of $\sigma$
which is defined by
\[
\W_{1, p} \sigma (x)
:=
\int_{0}^{\infty}
\left(
\frac{ \sigma(B(x, r)) }{r^{n - p}}
\right)^{\frac{1}{p - 1}}
\frac{dr}{r},
\quad x \in \R^{n}.
\]
Cao and Verbitsky \cite{MR3311903} showed that
there exists a unique finite energy solution
$u \in \dot{W}_{0}^{1, p}(\R^{n})$  to \eqref{eq:p-laplacian_homo}
under \eqref{intro:cond_sigma_1},
where $\dot{W}_{0}^{1, p}(\R^{n})$ is the homogeneous Sobolev space.
They also proved the necessity of \eqref{intro:cond_sigma_1}.
Seesanea and Verbitsky \cite{MR4048382}
extend such results to Eq. \eqref{eq:p-laplacian};
there exists
a unique finite energy solution $u \in \dot{W}_{0}^{1, p}(\R^{n})$ to \eqref{eq:p-laplacian}
if and only if \eqref {intro:cond_sigma_1} and
\begin{equation}\label{intro:cond_mu_1}
\int_{\R^{n}} \W_{1, p} \mu \, d \mu < \infty
\end{equation}
are fulfilled.
Treating general measure data $\mu \geq 0$ causes
problems about interaction between $\sigma$ and $\mu$.
The key to proof was to control them in the dual of $\dot{W}_{0}^{1, p}(\R^{n})$.

However, Eq. \eqref{eq:p-laplacian_homo}
has various infinite energy solutions.
In fact, in the classic paper by Brezis and Kamin \cite{MR1141779},
the existence and uniqueness of bounded solutions to \eqref{eq:p-laplacian_homo}
was proved under $p = 2$ and $\| \mathbf{I}_{2} \sigma \|_{L^{\infty}(\R^{n})} < \infty$.
Here, $ \mathbf{I}_{2} \sigma$ is the Newtonian potential of $\sigma$.
Their solutions do not belong to $\dot{W}_{0}^{1, 2}(\R^{n})$ in general.
Boccardo and Orsina \cite{MR1272564} treated elliptic equations with singular coefficients
and applied concept of \textit{renormalized solutions}.
Their solutions are also called \textit{$p$-superharmonic functions} in now.
For details of such generalized solutions,
see \cite{MR1205885, MR1264000, MR1354907, MR1409661, MR1760541, MR1955596, MR2305115, MR2859927}.

Recently, the study of generalized solutions to \eqref{eq:p-laplacian_homo}
has made significant progress.
Cao and Verbitsky \cite{MR3567503}
defined the \textit{intrinsic} Wolff potential $\mathbf{K}_{1, p, q} \sigma$ of $\sigma$
and proved that
there exists a minimal $p$-superharmonic solution to \eqref{eq:p-laplacian_homo}
if and only if the potentials $\W_{1, p} \sigma$ and $\mathbf{K}_{1, p, q} \sigma$
are not identically infinite.
Unfortunately, behavior of $\mathbf{K}_{1, p, q} \sigma$
can not be easily calculated from its definition.
Cao and Verbitsky \cite{MR3556326} constructed weak solutions in $W^{1, p}_{\loc}(\R^{n})$
under a certain capacity condition and
gave two-sided pointwise estimates of such solutions.
Seesanea and Verbitsky \cite{MR3985926} gave a sufficient condition for
the existence of $L^{r}$-integrable $p$-superharmonic solutions.
From existence of such solutions, behavior of the potentials is derived conversely. 
For very recent progress in the study of $\mathbf{K}_{1, p, q} \sigma$,
see \cite{MR4030348, VERBITSKY2019111516}.

In this paper, we extend results in \cite{MR3985926} to Eq. \eqref{eq:p-laplacian}.
We consider the following conditions:
\begin{equation}\label{cond:dsigma-wolff}
\int_{\R^{n}} (\W_{1, p} \sigma)^{\frac{(\gamma + q)(p - 1)}{p - 1 - q}} \, d \sigma
< \infty,
\end{equation}
\begin{equation}\label{cond:dmu-wolff}
\int_{\R^{n}} (\W_{1, p} \mu)^{\gamma} \, d \mu
< \infty,
\end{equation}
where $0 \leq \gamma < \infty$.
We denote by $\nu[u]$ the Riesz measure of a $p$-superharmonic function $u$
and interpret \eqref{eq:p-laplacian} as $\nu[u] = \sigma u^{q} + \mu$
(see Definition \ref{def:shs}).
Our main result is as follows.

\begin{Thm}\label{thm:main-p-laplacian}
Let $1 < p < n$, $0 < q < p - 1$.
Assume that \eqref{cond:dsigma-wolff} and \eqref{cond:dmu-wolff} hold
for some $0 < \gamma < \infty$ and that $(\sigma, \mu) \not\equiv (0, 0)$.
Then there exists a minimal $p$-superharmonic solution
$u$ to \eqref{eq:p-laplacian}.
Moreover, $u$ satisfies
\begin{equation}\label{generalized_energy}
\int_{\R^{n}} u^{\gamma} \, d \nu[u] < \infty
\end{equation}
and belongs to $L^{r, \rho}(\R^{n})$,
where $r = n(p - 1 + \gamma) / (n - p)$ and $\rho = p - 1 + \gamma$.
\end{Thm}

\begin{Rem}
(i)
Conversely, it follows from \cite[Theorem 2.3]{MR3311903} that
if there exists a $p$-superharmonic supersolution to \eqref{eq:p-laplacian}
satisfying \eqref{generalized_energy},
then \eqref{cond:dsigma-wolff} and \eqref{cond:dmu-wolff} 
must be fulfilled.
(ii)
As in \cite[Theorem 1.1]{MR3567503},
when $p \geq n$, there is no nontrivial supersolution to \eqref{eq:p-laplacian}.
\end{Rem}

Here, $L^{r, \rho}(\R^{n})$ denotes the Lorentz space with respect to the Lebesgue measure.
The authors do not know the same statement even if $\sigma = 0$.
Theorem \ref{thm:main-p-laplacian} includes the existence theorems in
\cite{MR3985926}  and \cite{MR4048382}
as the special cases $\mu = 0$ and $\gamma = 1$.
In general, our generalized solutions do not belong to $\dot{W}_{0}^{1, p}(\R^{n})$,
so we can not use the dual of $\dot{W}_{0}^{1, p}(\R^{n})$
to control interaction between $\sigma$ and $\mu$.
Hence, we derive an estimate of interaction directly using Wolff potentials
(see Theorem \ref{thm:mutual_energy_estimate}).
One of the authors used similar arguments for Green potentials in \cite{MR3881877}.
However, such arguments do not work for nonlinear potentials.
To overcome this difficulty, we use tools of nonlinear potential theory.
Theorem \ref{thm:mutual_energy_estimate}
can also be regarded as a generalization of \eqref{intro:trace_ineq}.
The Lorentz estimate for solutions is a direct consequence of it.

We also give variants of Theorem \ref{thm:main-p-laplacian}.
Theorem \ref{thm:main-p-laplacian_infty} and Proposition \ref{thm:main-p-laplacian_zero} 
are analogs of Theorem \ref{thm:main-p-laplacian} for $\gamma = \infty$ and $\gamma = 0$,
respectively.
In such cases, similar interaction between $\sigma$ and $\mu$
do not appear from difference of energy structures.
Theorem \ref{thm:main-p-laplacian_several_terms} is a generalization of
Theorem \ref{thm:main-p-laplacian} to equations of the form
\begin{equation}\label{eq:several_terms}
\begin{cases}
\displaystyle
- \laplacian_{p} u
=
\sum_{m = 1}^{M} \sigma^{(m)} u^{q_{m}} + \mu, \quad u > 0 \quad \text{in} \ \R^{n}, \\
\displaystyle
\liminf_{|x| \to \infty} u(x) = 0.
\end{cases}
\end{equation}
This result is new even if $p = 2$ and $\gamma = 1$.
However, the spirit of proof is the same as Theorem \ref{thm:main-p-laplacian}.
We also show the uniqueness of finite energy solutions.

\subsection*{Organization of the paper}
In Section \ref{sec:prelim}, we collect some facts of nonlinear potential theory to be used later.
In Section \ref{sec:energy_estimate}, we prove an estimate for mutual energy
and collect its consequences.
In Section \ref{sec:existence},
we prove Theorem \ref{thm:main-p-laplacian} using the results in the previous section.
In Sections \ref{sec:endpoint} and \ref{sec:several_terms},
we give some variants of Theorem \ref{thm:main-p-laplacian}.

\subsection*{Notation}
We use the following notation in this paper.
Let $\Omega$ be a domain (connected open subset) in $\R^{n}$. 
\begin{itemize}
\item
$B(x, R) := \{ y \in \R^{n} \colon |x - y| < R \}$.
\item
For $B = B(x, R)$ and $\lambda > 0$, we write $\lambda B := B(x, \lambda R)$. 
\item
$|A| :=$ the Lebesgue measure of a measurable set $A$.
\item
$\mathbf{1}_{A}(x) :=$ the indicator function of $A$.
\item
$C_{c}^{\infty}(\Omega) :=$
the set of all infinitely-differentiable functions with compact support in $\Omega$.
\item
$\M^{+}(\Omega) :=$ the set of all nonnegative Radon measure on $\Omega$.
\item
$A \approx B$ means $c_{1} A \leq B \leq c_{2} A$
for some constants $0 < c_{1} \leq c_{2} < \infty$ independent of $A$ and $B$.
\end{itemize}
For $\mu \in \M^{+}(\Omega)$, we denote by $L^{p}(\Omega, d \mu)$
the $L^{p}$ space with respect to $\mu$.
When $\mu$ is the Lebesgue measure,  
we write $L^{p}(\Omega, dx)$ as $L^{p}(\Omega)$ simply.
For a Banach space $X$, we denote by $X^{*}$ the dual of $X$.
We denote by $c$ and $C$ various constants with and without indices.

\section{Preliminaries}\label{sec:prelim}

\subsection{Function spaces}\label{sec:function_spaces}

Let $\Omega$ be a domain in $\R^{n}$, and let $1 < p < \infty$.
The Sobolev space $W^{1, p}(\Omega)$ ($W^{1, p}_{\loc}(\Omega)$) is
the space of all weakly differentiable functions $u$ such that
$u \in L^{p}(\Omega)$ and $|\nabla u| \in L^{p}(\Omega)$
($u \in L^{p}_{\loc}(\Omega)$ and $|\nabla u| \in L^{p}_{\loc}(\Omega)$).
The space $W_{0}^{1, p}(\Omega)$ 
is the closure of $C_{c}^{\infty}(\Omega)$ in
$W^{1, p}(\Omega)$.

We denote by $\dot{W}_{0}^{1, p}(\Omega)$
the set of all functions $u \in W^{1, p}_{\loc}(\Omega)$ such that
$|\nabla u| \in L^{p}(\Omega)$, and
$\| \nabla ( \varphi_{j} - u) \|_{L^{p}(\Omega)} \to 0$ as $j \to \infty$
for a sequence $\{ \varphi_{j} \}_{j = 1}^{\infty} \subset C_{c}^{\infty}(\Omega)$.
The space $\dot{W}_{0}^{1, p}(\Omega)$
is called the homogeneous Sobolev space (or Dirichlet space). 
When $1 < p < n$ or when $\Omega$ is bounded,
we define the norm of $\dot{W}_{0}^{1, p}(\Omega)$ by
$\| \nabla \cdot \|_{L^{p}(\Omega)}$. 
If $\Omega$ is bounded,
then $\dot{W}_{0}^{1, p}(\Omega) = W_{0}^{1, p}(\Omega)$
by the Poincar\'{e} inequality.
The following basic properties of functions in Sobolev spaces are parenthetically
used in our arguments.
Their proofs are similar to the ones of \cite[Theorems 1.18, 1.20 and 1.24]{MR2305115}.

\begin{Lem}\label{lem:cp_rule}
Suppose that $u \in \dot{W}_{0}^{1, p}(\Omega)$.
\begin{enumerate}[label=(\emph{\roman*})]
\item\label{lem:cp_rule_01}
Let $f \in C^{1}(\R)$ and $f(0) = 0$.
Assume that $f'$ is bounded on the range of $u$.
Then $f(u) \in \dot{W}_{0}^{1, p}(\Omega)$
and $\nabla f(u) = f'(u) \nabla u$ a.e. in $\Omega$.
\item\label{lem:cp_rule_02}
Let $w = \min\{ \max\{ u, m \}, M \}$,
where $m$ and $M$ are constants satisfying $m \leq 0 \leq M$.
Then $w \in \dot{W}_{0}^{1, p}(\Omega)$
and $\nabla  w = \mathbf{1}_{ \{ m < u < M \} } \nabla u$ a.e. in $\Omega$.
\item\label{lem:cp_rule_03}
Assume also that $u$ is bounded.
Suppose that $v \in W^{1, p}_{\loc}(\Omega) \cap L^{\infty}(\Omega)$
satisfies $|\nabla v|  \in L^{p}(\Omega)$.
Then $u v \in \dot{W}_{0}^{1, p}(\Omega)$
and $\nabla (uv) = v \nabla u + u \nabla v$ a.e. in $\Omega$.
\end{enumerate}
\end{Lem}

We also recall notion of Lorentz spaces \cite{MR2445437}.
\begin{Def}
Let $f$ be a measurable function on $\Omega$,
and let $0 < r, \rho \leq \infty$.
We define the \textit{Lorentz norm} of $f$ by
\[
\| f \|_{L^{r, \rho}(\Omega)}
=
\begin{cases}
\displaystyle
\left( \int_{0}^{\infty} \left( t^{\frac{1}{r}} f^{*}(t) \right)^{\rho} \frac{dt}{t} \right)^{\frac{1}{\rho}}
& \text{if} \ \rho < \infty, \\
\displaystyle
\sup_{t > 0} t^{\frac{1}{r}} f^{*}(t)
& \text{if} \ \rho = \infty,
\end{cases}
\]
where $f^{*}$ is the \textit{decreasing rearrangement} of $f$ which is defined by
\[
f^{*}(t) = \inf \{ \alpha > 0 \colon | \{ x \in \Omega \colon |f(x)| > \alpha \} | \leq t \}.
\]
The space of all $f$ with $\| f \|_{L^{r, \rho}(\Omega)} < \infty$
is denoted by $L^{r, \rho}(\Omega)$
and is called the \textit{Lorentz space} with indices $r$ and $\rho$.
\end{Def}

\subsection{$p$-Laplacian and $p$-superharmonic functions}\label{sec:nlpt}

For $u \in W^{1, p}_{\loc}(\Omega)$,
we define the $p$-Laplacian $\laplacian_{p}$
in the weak (distributional) sense, i.e., for every $\varphi \in C_{c}^{\infty}(\Omega)$,
\[
\langle - \laplacian_{p} u, \varphi \rangle
=
\int_{\Omega} |\nabla u|^{p - 2} \nabla u \cdot \nabla \varphi \, dx.
\]
A function $u \in W^{1, p}_{\loc}(\Omega)$ is called as
\textit{$p$-harmonic} if $u$ is a continuous weak solution to
\begin{equation}\label{eqn:p-harmonic}
- \laplacian_{p} u = 0 \quad \text{in} \ \Omega.
\end{equation}

For basic properties of the $p$-Laplacian including
comparison principles for weak solutions
and solvability of Dirichlet problems,
we refer to \cite[Chapter 3]{MR2305115} and \cite{MR567696}.

To treat measure data problems, we introduce $p$-superharmonic functions.
A function $u \colon \Omega \to ( - \infty, \infty]$ is called \textit{$p$-superharmonic} if
$u$ is lower semicontinuous in $\Omega$, is not identically infinite in any component of $\Omega$,
and $u$ satisfies the comparison principle on each subdomain $D \Subset \Omega$;
if $h \in C(\cl{D})$ is $p$-harmonic in $D$
and if $u \geq h$ on $\del D$, then $u \geq h$ in $D$.

By \cite[Theorem 7.22]{MR2305115},
if $u$ and $v$ are $p$-superharmonic in $\Omega$ and if $u \leq v$ a.e. in $\Omega$,
then $u(x) \leq v(x)$ for all $x \in \Omega$.
If $u \in W^{1, p}_{\loc}(\Omega)$ is a supersolution to \eqref{eqn:p-harmonic},
then it has a lower semicontinuous representative
and can be regarded as a $p$-superharmonic function
up to taking such a representative
(see \cite[Theorems 3.63 and 7.25]{MR2305115}).
If $u$ is a $p$-superharmonic function in $\Omega$,
then its truncation $\min\{ u, k \}$
is a supersolution to \eqref{eqn:p-harmonic} for each $k > 0$.
Hence, there exists a unique Radon measure $\nu[u]$ such that
\[
\int_{\Omega} |\D u|^{p - 2} \D u \cdot \nabla \varphi \, dx
=
\int_{\Omega} \varphi \, d \nu[u]
\quad \forall \varphi \in C_{c}^{\infty}(\Omega),
\]
where $\D u$ is the \textit{very weak gradient} of $u$ which is defined by
\[
\D u := \lim_{k \to \infty} \nabla \min\{ u, k \}.
\]
The measure $\nu[u]$ is called the \textit{Riesz measure} of $u$.
By definition, if $u \in W^{1, p}_{\loc}(\Omega)$,
then $\D u = \nabla u$ and $\nu[u] = - \laplacian_{p} u$
in the sense of weak solutions.

We say that a property holds \textit{quasieverywhere} (q.e.)
if it holds except on a set of $p$-capacity zero.
Here, for $E \subset \R^{n}$,
the (Sobolev) $p$-capacity is defined by
\[
C_{p}(E)
=
\inf
\int_{\R^{n}} (|u|^{p} + |\nabla u|^{p}) \, dx,
\]
where the infimum is taken over all
$u \in W^{1, p}(\R^{n})$
such that $u = 1$ in a neighborhood of $E$.
We note that every $u \in W^{1, p}_{\loc}(\Omega)$
has a \textit{quasicontinuous} representative,
which coincides with $u$  quasieverywhere
and that every $p$-superharmonic function $u$ is quasicontinuous
(see, e.g., \cite[Theorems 4.4 and 10.9]{MR2305115}).
Henceforth, we assume that $u$ is always chosen to be quasicontinuous.

\subsection{Wolff potentials}\label{sec:wolff_potential}

\begin{Def}
Let $1 < p < \infty$.
The Wolff potential $\W_{1, p} \sigma$ of $\sigma \in \M^{+}(\R^{n})$ is defined by
\[
\W_{1, p} \sigma (x)
:=
\int_{0}^{\infty}
\left(
\frac{ \sigma(B(x, r)) }{r^{n - p}}
\right)^{\frac{1}{p - 1}}
\frac{dr}{r}.
\]
\end{Def}

This nonlinear potential was first introduced by 
Havin and Maz'ya \cite{MR0409858}.
For any $\sigma \in \M^{+}(\R^{n})$,
$\W_{1, p} \sigma(x)$ is a lower semicontinuous function of $x$
(see \cite{MR727526}).
By a simple calculation,
for any $\sigma, \mu \in \M^{+}(\R^{n})$ and $\gamma, \beta \geq 0$,
\begin{equation}\label{quasi-linearity}
\W_{1, p}( \gamma \sigma + \beta \mu )(x)
\leq
c(p)
\left(
\gamma^{\frac{1}{p - 1}} \W_{1, p}\sigma(x)
+
\beta^{\frac{1}{p - 1}} \W_{1, p}\mu(x)
\right)
\quad
\forall x \in \R^{n}.
\end{equation}
Also, the following weak maximum principle for Wolff potentials holds:
\begin{equation}\label{wolff-max}
\W_{1, p}\sigma(x)
\leq
c \sup_{\spt \sigma} \W_{1, p}\sigma
\quad
\forall x \in \R^{n},
\end{equation}
where $c = c(n, p) > 0$
(see, e.g., \cite{VERBITSKY2019111516}).
Moreover, by the lower semicontinuity of $\W_{1, p}\sigma$,
we observe that $\sup_{\spt \sigma} \W_{1, p}\sigma = \| \W_{1, p}\sigma \|_{L^{\infty}(\R^{n}, d \sigma)}$.

It was shown in \cite[Theorem 1.11]{MR1747901} that
for any $\sigma \in \M^{+}(\R^{n})$,
\[
\begin{split}
\int_{K}  \frac{d \sigma}{ (\W_{1, p}\sigma)^{p - 1} }
\leq
C \capacity_{p}(K, \R^{n})
\end{split}
\]
for any compact set $K \subset \R^{n}$,
where $C = C(n, p)$ is a constant
and $\capacity_{p}(K, \R^{n})$ is the variational $p$-capacity of $(K, \R^{n})$.
From this inequality, one can easily deduce 
(by using a similar argument in \cite[Lemma 3.6]{MR3567503})
that \eqref{cond:dsigma-wolff} implies 
$\sigma$ must be absolutely continuous with respect to the $p$-capacity,
that is,
$\sigma(E) = 0$ whenever
$C_{p}(E) = 0$
for every Borel set $E \subset \R^{n}$. 

The following two-sided Wolff potential bounds
were established by Kilpel{\"a}inen and Mal{\'y} \cite{MR1205885,MR1264000}.

\begin{Thm}[{\cite[Theorem 1.6]{MR1264000}}]\label{pointwise_est_p-superharmonic} 
Let $1 < p < \infty$.
Suppose that $u$ is a nonnegative $p$-superharmonic function in $2B = B(x, 2R)$
and that $\mu$ is the Riesz measure of $u$.
Then
\[
\frac{1}{c_{K}} \W_{1, p}^{R} \mu(x)
\leq
u(x)
\leq
c_{K} \left(
\inf_{B} u
+
\W_{1, p}^{2R} \mu(x)
\right),
\] 
where $c_{K} = c_{K}(n, p) \geq 1$ and
$\W_{1, p}^{R} \mu$ is the \textit{truncated} Wolff potential of $\mu$ which is defined by
\[
\W_{1, p}^{R} \mu (x)
:=
\int_{0}^{R}
\left(
\frac{ \mu(B(x, r)) }{r^{n - p}}
\right)^{\frac{1}{p - 1}}
\frac{dr}{r}.
\]
\end{Thm}

\section{Estimate for mutual energy and its consequences}\label{sec:energy_estimate}

The following Wolff energy estimate is our key ingredient.

\begin{Thm}\label{thm:mutual_energy_estimate}
Let $1 < p < n$, $0 < \gamma < \infty$ and $- \gamma < q < p - 1$.
Then for any $\mu, \sigma \in \M^{+}(\R^{n})$,
\[
\begin{split}
& 
\int_{\R^{n}} (\W_{1, p} \mu)^{\gamma + q} \, d \sigma \\
& \leq
C
\left(
\int_{\R^{n}} (\W_{1, p} \mu)^{\gamma} \, d \mu
\right)^{\frac{\gamma + q}{p - 1 + \gamma}}
\left(
\int_{\R^{n}} (\W_{1, p} \sigma)^{\frac{(\gamma + q)(p - 1)}{p - 1 - q}} \, d \sigma
\right)^{\frac{p - 1 - q}{p - 1 + \gamma}},
\end{split}
\]
where $C$ is a positive constant depending only on $n$, $p$, $\gamma$ and $q$.
\end{Thm}

To derive this estimate,
we prove the following simple lemma.

\begin{Lem}\label{lem:log_caccioppoli}
Let $1 < p < \infty$.
Let $u \in \dot{W}_{0}^{1, p}(\Omega)$,
and let $v \in W^{1, p}_{\loc}(\Omega)$
be a nonnegative $p$-superharmonic function in $\Omega$.
Assume also that $\| \nabla v \|_{L^{p}(\Omega)}$ and $\nu[v](\Omega)$ are finite.
Then
\[
\int_{\Omega}
|u|^{p} v^{1 - p} \, d \nu[v]
\leq
\int_{\Omega}
|\nabla u|^{p}
\, dx.
\]
\end{Lem}

\begin{proof}
We set $u_{+}^{M} = \min\{ u_{+}, M \}$ and $v_{M} = v + M^{-1}$,
where $u_{+} = \max\{ u, 0 \}$ and $M$ is a positive constant.
Then
$(u_{+}^{M})^{p} (v_{M})^{1 - p} \in \dot{W}_{0}^{1, p}(\Omega)$
by Lemma \ref{lem:cp_rule}.
Since $\D v = \nabla v = \nabla v_{M} \in L^{p}(\Omega)$,
by density arguments as in \cite[Lemma 21.14]{MR2305115},
\[
\int_{\Omega}
\varphi
\, d \nu[v]
=
\int_{\Omega}
|\nabla v_{M}|^{p - 2} \nabla v_{M}
\cdot 
\nabla \varphi
\, dx
\quad \forall \varphi \in \dot{W}_{0}^{1, p}(\Omega).
\]
Substituting $(u_{+}^{M})^{p} (v_{M})^{1 - p}$ into $\varphi$
and
using a Picone type inequality
(see \cite[Theorem 1.1]{MR1618334}), we get
\[
\int_{\Omega} (u_{+}^{M})^{p} (v_{M})^{1 - p} \, d \nu[v]
\leq
\int_{\Omega} |\nabla (u_{+}^{M}) |^{p} \, dx.
\]
Taking the limit $M \to \infty$, we arrive at
\[
\int_{\Omega}
u_{+}^{p} v^{1 - p} \, d \nu[v]
\leq
\int_{\Omega}
|\nabla u_{+}|^{p}
\, dx.
\]
Applying the same argument to
$u_{-} = (-u)_{+}$, we get the desired estimate.
\end{proof}

\begin{Lem}\label{lem:trace_estimate}
Let $1 < p < \infty$, $0 < \gamma < \infty$ and $- \gamma < q < p - 1$.
Let $u \in \dot{W}_{0}^{1, p}(\Omega) \cap L^{\infty}(\Omega)$,
and let $v \in W^{1, p}_{\loc}(\Omega)$
be a nonnegative $p$-superharmonic function in $\Omega$.
Assume also that $\| \nabla v \|_{L^{p}(\Omega)}$ and $\nu[v](\Omega)$ are finite.
Then
\[
\begin{split}
& \int_{\Omega} |u|^{\gamma + q} \, d \nu[v] \\
& \leq
C
\left(
\int_{\Omega} |\nabla u|^{p} |u|^{\gamma - 1} \, dx
\right)^{\frac{\gamma + q}{p - 1 + \gamma}}
\left(
\int_{\Omega} v^{\frac{(\gamma + q)(p - 1)}{p - 1 - q}} \, d \nu[v]
\right)^{\frac{p - 1 - q}{p - 1 + \gamma}},
\end{split}
\]
where $C$ is a positive constant depending only on $p$, $\gamma$ and $q$.
\end{Lem}

\begin{proof}
Without loss of generality, 
we may assume that both integrals on the right-hand side are finite
and that $u \geq 0$.
Applying H\"{o}lder's inequality to $d \omega = v^{1 - p} d \nu[v]$, we get
\begin{equation}\label{trace_estimate_1}
\begin{split}
&
\int_{\Omega}
u^{\gamma + q}
\, d \nu[v]
=
\int_{\Omega}
u^{\gamma + q}
v^{p - 1}
\, d \omega \\
& \leq
\left(
\int_{\Omega}
u^{p - 1 + \gamma}
\, d \omega
\right)^{\frac{\gamma + q}{p - 1 + \gamma}}
\left(
\int_{\Omega} v^{\frac{(p - 1 + \gamma)(p - 1)}{p - 1 - q}} \, d \omega
\right)^{\frac{p - 1 - q}{p - 1 + \gamma}}.
\end{split}
\end{equation}
For each $\epsilon > 0$, we set
$w_{\epsilon} = (u^{\frac{p - 1 + \gamma}{p}} - \epsilon)_{+}$.
Then by Lemmas \ref{lem:cp_rule} and \ref{lem:log_caccioppoli},
\[
\int_{\Omega}
w_{\epsilon}^{p} v^{1 - p} \, d \nu[v]
\leq
\int_{\Omega}
|\nabla w_{\epsilon}|^{p}
\, dx.
\]
Since
\[
\nabla w_{\epsilon}
=
\frac{p - 1 + \gamma}{p} \nabla u u^{\frac{\gamma - 1}{p}}
\mathbf{1}_{ \{ u^{\frac{p - 1 + \gamma}{p}} > \epsilon \} }
\quad \text{a.e. in} \ \Omega,
\]
by the monotone convergence theorem,
\begin{equation}\label{trace_estimate_2}
\begin{split}
\int_{\Omega}
u^{p - 1 + \gamma} v^{1 - p} \, d \nu[v]
& \leq
\left( \frac{p - 1 + \gamma}{p} \right)^{p}
\int_{\Omega}
|\nabla u|^{p} u^{\gamma - 1}
\, dx.
\end{split}
\end{equation}
From \eqref{trace_estimate_1} and \eqref{trace_estimate_2},
we obtain the desired estimate.
\end{proof}

Let $1 < p < n$.
For $k \in \N$,
we set $\mu_{k} = \mathbf{1}_{\Omega(\mu, k)} \mu$, where
\begin{equation}\label{cutoff_set}
\Omega(\mu, k)
=
\left\{
x \in \R^{n} \colon \W_{1, p} \mu(x) \leq k
\right\}
\cap
\overline{B(0, 2^{k})}.
\end{equation}
Then
\[
\int_{\R^{n}} \W_{1, p} \mu_{k} \, d \mu_{k}
=
\int_{\Omega(\mu, k)} \W_{1, p} \mu_{k} \, d \mu
\leq
k \mu( \Omega(\mu, k) ) < \infty
\]
which is equivalent to
$\mu_{k} \in (\dot{W}_{0}^{1, p}(\R^{n}))^{*}$
by the Hedberg-Wolff theorem 
(see \cite{MR727526} or \cite[Theorem 4.5.4]{MR1411441}).
Thus,
there exists a unique nonnegative $p$-superharmonic function
$u_{k} \in \dot{W}_{0}^{1, p}(\R^{n})$
satisfying $- \laplacian_{p} u_{k} = \mu_{k}$ in $\R^{n}$
and $\liminf_{|x| \to \infty} u_{k}(x) = 0$.
Moreover, by Theorem \ref{pointwise_est_p-superharmonic}
and \eqref{wolff-max},
\[
\| u \|_{L^{\infty}(\R^{n})}
\leq
c_{K} \| \W_{1, p} \mu_{k} \|_{L^{\infty}(\R^{n})}
\leq
C k.
\]

\begin{Rem}\label{rem:approximation_of_wolff_potential}
For $\mu \in \M^{+}(\R^{n})$,
the $p$-capacity of $\{ x \in \R^{n} \colon \W_{1, p}\mu(x) = \infty \}$ is zero
by Theorem \ref{pointwise_est_p-superharmonic}
(see \cite[Remark 3.7]{MR1402674}).
Thus, if $\mu$ is absolutely continuous with respect to
the $p$-capacity,
then
$\mathbf{1}_{\Omega(\mu, k)} \uparrow \mathbf{1}_{\R^{n}}$ $d \mu$-a.e.,
and
$\W_{1, p} \mu_{k}(x) \uparrow \W_{1, p} \mu(x)$
for all $x \in \R^{n}$.
\end{Rem}

\begin{proof}[Proof of Theorem \ref{thm:mutual_energy_estimate}]
We may assume that both integrals on the right-hand side are finite without loss of generality.
Thus,
$\mu$ and $\sigma$ are absolutely continuous with respect to the $p$-capacity.
For each $k \in \N$, put $\mu_{k} = \mathbf{1}_{\Omega(\mu, k)} \mu$ and
$\sigma_{k} = \mathbf{1}_{ \Omega(\sigma, k) } \sigma$, 
where $\Omega(\mu, k)$ and $\Omega(\sigma, k)$
are defined by \eqref{cutoff_set}.
Let $u_{k}, v_{k} \in \dot{W}_{0}^{1, p}(\R^{n})$
be the bounded finite energy $p$-superharmonic functions satisfying
$-\laplacian_{p} u_{k}  = \mu_{k}$ and
$-\laplacian_{p} v_{k}  = \sigma_{k}$ in $\R^{n}$,
respectively.
By Theorem \ref{pointwise_est_p-superharmonic},
$u_{k} \approx \W_{1, p} \mu_{k}$ and 
$v_{k} \approx \W_{1, p} \sigma_{k}$ in $\R^{n}$.
By Lemma \ref{lem:cp_rule},
$(u_{k}^{\gamma} - \epsilon)_{+} \in \dot{W}_{0}^{1, p}(\R^{n})$ for any $\epsilon > 0$.
Since $\mu_{k}$ is the Riesz measure of $u_{k}$,
\[
\begin{split}
\int_{\R^{n}} (u_{k}^{\gamma} - \epsilon)_{+} \, d \mu_{k}
& =
\int_{\R^{n}} 
|\nabla u_{k}|^{p - 2} \nabla u_{k} \cdot \nabla (u_{k}^{\gamma} - \epsilon)_{+}
\, dx \\
& =
\gamma \int_{ \{ u_{k}^{\gamma} > \epsilon \} } |\nabla u_{k}|^{p} u_{k}^{\gamma - 1} \, dx.
\end{split}
\]
By the monotone convergence theorem,
\[
\begin{split}
\int_{\R^{n}} u_{k}^{\gamma} \, d \mu_{k}
& =
\lim_{\epsilon \to 0}
\int_{\R^{n}} (u_{k}^{\gamma} - \epsilon)_{+} \, d \mu_{k} \\
& =
\lim_{\epsilon \to 0}
\gamma \int_{ \{ u_{k}^{\gamma} > \epsilon \} } |\nabla u_{k}|^{p} u_{k}^{\gamma - 1} \, dx
=
\gamma \int_{ \R^{n} }  |\nabla u_{k}|^{p} u_{k}^{\gamma - 1} \, dx.
\end{split}
\]
Consequently, we have the following estimates:
\begin{align*}
\int_{\R^{n}} \left( \W_{1, p} \mu_{k} \right)^{\gamma + q} \, d \sigma_{k}
& \approx
\int_{\R^{n}} u_{k}^{\gamma + q} \, d \sigma_{k}, \\
\int_{\R^{n}} v_{k}^{\frac{(\gamma + q)(p - 1)}{p - 1 - q}} \, d \sigma_{k}
& \approx
\int_{\R^{n}} \left( \W_{1, p} \sigma_{k} \right)^{\frac{(\gamma + q)(p - 1)}{p - 1  - q}} \, d \sigma_{k}, \\
\int_{ \R^{n} }  |\nabla u_{k}|^{p} u_{k}^{\gamma - 1} \, dx
& \approx
\int_{\R^{n}} \left( \W_{1, p} \mu_{k} \right)^{\gamma} \, d \mu_{k}.
\end{align*}
Here, the constants in equivalence depend only on $n$, $p$, $\gamma$ and $q$.
Combining these estimates and Lemma \ref{lem:trace_estimate}, we get
\[
\begin{split}
&
\int_{\R^{n}} (\W_{1, p} \mu_{k})^{\gamma + q} \, d \sigma_{k} \\
&
\leq
C
\left(
\int_{\R^{n}} (\W_{1, p} \mu_{k})^{\gamma} \, d \mu_{k}
\right)^{\frac{\gamma + q}{p - 1 + \gamma}}
\left(
\int_{\R^{n}} (\W_{1, p} \sigma_{k})^{\frac{(\gamma + q)(p - 1)}{p - 1 - q}} \, d \sigma_{k}
\right)^{\frac{p - 1 - q}{p - 1 + \gamma}}.
\end{split}
\]
By Remark \ref{rem:approximation_of_wolff_potential},
$\W_{1, p} \mu_{k}(x) \uparrow \W_{1, p} \mu(x)$ for all $x \in \R^{n}$ and
$\mathbf{1}_{\Omega(\sigma, k)} \uparrow \mathbf{1}_{\R^{n}}$ $d \sigma$-a.e.
Therefore,
\[
(\W_{1, p} \mu_{k})^{\gamma + q} \mathbf{1}_{\Omega(\sigma, k)}
\uparrow
(\W_{1, p} \mu)^{\gamma + q}
\quad \text{$d \sigma$-a.e.}
\]
Using the monotone convergence theorem, we arrive at the desired estimate.
\end{proof}

The following quasi-triangle inequality is a direct consequence of
Theorem \ref{thm:mutual_energy_estimate}.
When $\gamma = 1$, it readily follows from 
the Hedberg-Wolff theorem.

\begin{Cor}\label{cor:quasi-linearity_of_energy_space}
Let $1 < p < n$ and $0 < \gamma < \infty$.
Then for any $\mu, \nu \in \M^{+}(\R^{n})$,
\[
\begin{split}
&
\int_{\R^{n}} (\W_{1, p} (\mu + \nu) )^{\gamma} \, d (\mu + \nu) \\
& \approx
\int_{\R^{n}} (\W_{1, p} \mu )^{\gamma} \, d \mu
+
\int_{\R^{n}} (\W_{1, p} \nu )^{\gamma} \, d \nu,
\end{split}
\]
where the constants in equivalence depend only on $n$, $p$ and $\gamma$.
\end{Cor}

\begin{proof}
Each of the right-hand side is controlled by the left-hand side.
Let us estimate the left-hand side.
By \eqref{quasi-linearity},
\[
\begin{split}
&
\int_{\R^{n}} (\W_{1, p} (\mu + \nu) )^{\gamma} \, d (\mu + \nu) \\
& \leq
C \int_{\R^{n}} (\W_{1, p} \mu)^{\gamma} \, d \mu
+
C \int_{\R^{n}} (\W_{1, p} \nu)^{\gamma} \, d \nu \\
& +
C \int_{\R^{n}} (\W_{1, p} \mu)^{\gamma} \, d \nu
+
C \int_{\R^{n}} (\W_{1, p} \nu)^{\gamma} \, d \mu.
\end{split}
\]
Applying Theorem \ref{thm:mutual_energy_estimate} with $q = 0$,
we can estimate the latter two terms by other two.
Then the assertion follows from Young's inequality.
\end{proof}

\begin{Cor}\label{cor:lorentz_integrability_of_wolff_potentials}
Let $1 < p < n$ and $0 < \gamma < \infty$.
Then
\begin{equation}
\label{eqn:sobolev@cor:lorentz_integrability_of_wolff_potentials}
\| \W_{1, p} \mu \|_{L^{r, \rho}(\R^{n})}
\leq
C
\left(
\int_{\R^{n}} (\W_{1, p} \mu)^{\gamma} \, d \mu
\right)^{\frac{1}{p - 1 + \gamma}}
\end{equation}
for any $\mu \in \M^{+}(\R^{n})$,
where 
$r = n (p - 1 + \gamma) / (n - p)$, $\rho = p - 1 + \gamma$
and $C$ is a positive constant depending only on $n$, $p$ and $\gamma$.
\end{Cor}

\begin{proof}
It is known that for any $\sigma \in \M^{+}(\R^{n})$,
\begin{equation*}\label{inq:HM}
\W_{1, p}\sigma(x)
\leq
c(n, p) \mathbf{V}_{1, p}\sigma(x)
\end{equation*}
for all $x \in \R^{n}$,
where
$\mathbf{V}_{1, p}\sigma := \I_{1}[(\I_{1}\sigma)^{\frac{1}{p - 1}} dx]$
is the Havin-Maz'ya potential of $\sigma$
(see \cite{MR727526}).
Therefore, by boundedness of Riesz potentials
(see, e.g., \cite[Theorem 1.4.19]{MR2445437}),
\[
\| \W_{1, p}(f dx) \|_{L^{r, \rho}(\R^{n})}
\leq
C \| f \|_{L^{\frac{r}{r - \gamma}, \frac{\rho}{\rho - \gamma}}(\R^{n})}^{\frac{1}{p - 1}}
\]
for any nonnegative $f \in L^{\frac{r}{r - \gamma}, \frac{\rho}{\rho - \gamma}}(\R^{n})$.
Hence, by H\"{o}lder's inequality,
\begin{equation}
\label{eqn:energy_of_fdx@cor:lorentz_integrability_of_wolff_potentials}
\begin{split}
\int_{\R^{n}} (\W_{1, p}(f dx))^{\gamma} f \, dx
& \leq
\| (\W_{1, p}(f dx))^{\gamma} \|_{L^{\frac{r}{\gamma}, \frac{\rho}{\gamma}}(\R^{n})}
\| f \|_{L^{\frac{r}{r - \gamma}, \frac{\rho}{\rho - \gamma}}(\R^{n})} \\
& =
\| \W_{1, p}(f dx) \|_{L^{r, \rho}(\R^{n})}^{\gamma}
\| f \|_{L^{\frac{r}{r - \gamma}, \frac{\rho}{\rho - \gamma}}(\R^{n})} \\
& \leq
C \| f \|_{L^{\frac{r}{r - \gamma}, \frac{\rho}{\rho - \gamma}}(\R^{n})}^{1 + \frac{\gamma}{p - 1}}.
\end{split}
\end{equation}
Combining \eqref{eqn:energy_of_fdx@cor:lorentz_integrability_of_wolff_potentials}
and
Theorem \ref{thm:mutual_energy_estimate},
we get
\begin{equation*}\label{mutual_energy_estimate_of_mu_and_fdx}
\int_{\R^{n}} (\W_{1, p}\mu)^{\gamma} f \, dx
\leq
C 
\left(
\int_{\R^{n}} (\W_{1, p} \mu)^{\gamma} \, d \mu
\right)^{\frac{\gamma}{p - 1 + \gamma}}
\| f \|_{L^{\frac{r}{r - \gamma}, \frac{\rho}{\rho - \gamma}}(\R^{n})}.
\end{equation*}
Thus, by a dual characterization of $L^{\frac{r}{\gamma}, \frac{\rho}{\gamma}}(\R^{n})$,
\begin{equation*}
\begin{split}
&
\| \W_{1, p} \mu \|_{L^{r, \rho}(\R^{n})}^{\gamma}
=
\| (\W_{1, p} \mu)^{\gamma} \|_{L^{\frac{r}{\gamma}, \frac{\rho}{\gamma}}(\R^{n})}
\\
& \leq
C \sup \left\{
\int_{\R^{n}} (\W_{1, p}\mu)^{\gamma} f \, dx
\colon
\| f \|_{L^{\frac{r}{r - \gamma}, \frac{\rho}{\rho - \gamma}}(\R^{n})} \leq 1,
f \geq 0
\right\}
\\
& \le
C
\left(
\int_{\R^{n}} (\W_{1, p} \mu)^{\gamma} \, d \mu
\right)^{\frac{\gamma}{p - 1 + \gamma}}.
\end{split}
\end{equation*}
This completes the proof.
\end{proof}

\begin{Rem}
The inequality \eqref{eqn:energy_of_fdx@cor:lorentz_integrability_of_wolff_potentials}
gives a refinement of \cite[Corollary 1.2]{MR3985926}.
Suppose that $1 < p < n$, $0 < q < p - 1$ and $0 < \gamma < \infty$.
Let $f$ and $g$ be nonnegative functions on $\R^{n}$.
\begin{enumerate}
\item
If $d \sigma = f dx$ and if
\[
f \in L^{s, t}(\R^{n})
\ \text{with} \
s = \frac{n(p - 1 + \gamma)}{n (p - 1 - q) + p(\gamma + q)}, \
t = \frac{p - 1 + \gamma}{p - 1 - q},
\]
then \eqref{cond:dsigma-wolff} is fulfilled. 
\item
If $d \mu = g dx$ and if
\[
g \in L^{s, t}(\R^{n})
\ \text{with} \
s = \frac{n(p - 1 + \gamma)}{n(p - 1) + p \gamma}, \
t = \frac{p - 1 + \gamma}{p - 1},
\]
then \eqref{cond:dmu-wolff} is fulfilled.
\end{enumerate}
Hence, in view of Theorem \ref{thm:main-p-laplacian},
there exists a minimal $p$-superharmonic solution $u \in L^{r, \rho}(\R^{n})$ to \eqref{eq:p-laplacian},
where $r = n (p - 1 + \gamma) / (n - p)$ and $\rho = p - 1 + \gamma$.
\end{Rem}

\section{Construction of minimal solutions to \eqref{eq:p-laplacian}}\label{sec:existence}

Throughout, we assume that $1 < p < n$ and $0 < q < p - 1$.
Our definition of generalized solutions is as follows:

\begin{Def}\label{def:shs}
Let $\Omega$ be a domain in $\R^{n}$, and let $\sigma, \mu \in \M^{+}(\Omega)$.
A nonnegative function $u$ is said to be a
\textit{$p$-superharmonic solution} (\textit{supersolution}) to the equation
\begin{equation*}\label{def:p-shs}
- \laplacian_{p} u = \sigma u^{q} + \mu \quad \text{in} \ \Omega
\end{equation*}
if $u$ is $p$-superharmonic in $\Omega$,
$u \in L^{q}_{\loc}(\Omega, d \sigma)$,
and $\nu[u] = u^{q} d \sigma + \mu$
($\nu[u] \geq u^{q} d \sigma + \mu$).
We say that a nontrivial $p$-superharmonic solution $u$ to \eqref{eq:p-laplacian} is \textit{minimal}
if  $w \geq u$ in $\R^{n}$
whenever $w$ is a nontrivial $p$-superharmonic supersolution to \eqref{eq:p-laplacian}.
\end{Def}

The following theorem was established by
Cao and Verbitsky \cite{MR3311903, MR3567503}.
It gives pointwise lower estimates of supersolutions to \eqref{eq:p-laplacian_homo}.
\begin{Thm}[{\cite[Theorem 2.3]{MR3311903}}]
\label{thm:lower_bound_of_sol_to_subeq}
Let $w \in L^{q}_{\loc}(\R^{n}, d \sigma)$ be
a positive $p$-superharmonic supersolution to \eqref{eq:p-laplacian_homo}.
Then
\[
w \geq c_{0} (\W_{1, p} \sigma)^{ \frac{p - 1}{p - 1 - q} },
\] 
where $c_{0}$ is a constant depending only on $n$, $p$ and $q$.
\end{Thm}

To construct minimal solutions to \eqref{eq:p-laplacian},
we consider a family of solutions to localized problems
and solve the localized problems using a sub- and supersolution method.
We shall need the following weighted norm inequality.

\begin{Lem}[{\cite[Lemma 2.1]{MR3985926}}]
\label{lem:weighted_norm_inequality}
Assume that \eqref{cond:dsigma-wolff} holds with $0 < \gamma < \infty$.
Then for any $f \in L^{\frac{\gamma + q}{q}}(\R^{n}, d \sigma)$,
\begin{equation*}\label{WN_ineq1}
\| \W_{1, p}(f d \sigma) \|_{L^{\gamma + q}(\R^{n}, d \sigma)}
\leq
C \| f \|_{L^{\frac{\gamma + q}{q}}(\R^{n}, d \sigma)}^{\frac{1}{p - 1}},
\end{equation*}
where
$C$ is a constant depending only on
$n$, $p$, $q$, $\gamma$ and the upper bound of \eqref{cond:dsigma-wolff}.
\end{Lem}

\begin{Lem}\label{lem:main-p-laplacian}
Let $B = B(x_{0}, R)$ be a ball in $\R^{n}$.  Assume that
\eqref{cond:dsigma-wolff} and \eqref{cond:dmu-wolff} hold with $0 < \gamma < \infty$.
Assume also that
\[
\sup_{\spt \sigma} \W_{1, p}\sigma < \infty, \quad \spt \sigma \subset \cl{B}
\]
and
\[
\sup_{\spt \mu} \W_{1, p}\mu < \infty, \quad \spt \mu \subset \cl{B}.
\]
Then there exists a nonnegative $p$-superharmonic function
$u$ satisfying
\begin{equation}\label{eq:p-laplace_B}
- \laplacian_{p} u = \sigma u^{q} + \mu \quad \text{in} \ 2B
\end{equation}
and
\begin{equation*}
\min\{ u, l \} \in \dot{W}_{0}^{1, p}(2B) \quad \forall l > 0.
\end{equation*}
Moreover, $u$ satisfies the following properties:
\begin{enumerate}[label=(\emph{\roman*})]
\item\label{norm_bounds_for_sols_to_lp}
The solution $u$ belongs to
$L^{\gamma + q}(2B, d \sigma) \cap L^{r, \rho}(2B)$.
Moreover,
\[
\left(
\int_{2B}
(\W_{1, p}(u^{q} d \sigma))^{\gamma}
u^{q} \, d \sigma
\right)^{\frac{1}{p - 1 + \gamma}}
+
\| u \|_{L^{r, \rho}(2B)}
\leq
C,
\]
where $C$ is a constant depending only on
$n$, $p$, $q$, $\gamma$ and
the bounds of \eqref{cond:dsigma-wolff} and \eqref{cond:dmu-wolff}.
\item\label{lower_bound_of_sols_to_lp}
For every $x \in B$,
\[
u(x)
\geq
\frac{1}{C}
\left\{
(\W_{1, p}^{\frac{R}{2}} \sigma(x))^{ \frac{p - 1}{p - 1 - q} } + \W_{1, p}^{\frac{R}{2}} \mu(x)
\right\},
\]
where $C$ is a positive constant depending only on $n$, $p$ and $q$.
\item\label{upper_bound_of_sols_to_lp}
If $w$ is a $p$-superharmonic supersolution to \eqref{eq:p-laplacian},
then $w \geq u$ in 2B.
\item\label{comparison_principle_for_sols_to_lp}
Let $u_{1}$ be a $p$-superharmonic function which defined by \eqref{definition_of_u_j}.
Assume that $w$ is a nonnegative $p$-superharmonic supersolution to \eqref{eq:p-laplace_B}
and that $w \geq u_{1}$ in $2B$.
Then $w \geq u \geq u_{1}$ in $2B$.
\end{enumerate}
\end{Lem}

\begin{proof}
\noindent
\textbf{Step 1.}
We construct approximate solutions $\{ u_{j} \}_{j = 1}^{\infty}$.
By assumptions on $\sigma$ and $\mu$,
along with the Hedberg-Wolff theorem (\cite[Theorem 4.5.4]{MR1411441}),
$\sigma$ and $\mu$ belong to $(\dot{W}_{0}^{1, p}(2B))^{*}$.
Therefore, there exists a bounded finite energy $p$-superharmonic function
$v \in \dot{W}_{0}^{1, p}(2B) \cap L^{\infty}(2B)$ satisfying
\[
- \laplacian_{p} v
=
\tilde{\sigma}
:=
\left( \frac{p - 1 - q}{p - 1} \right)^{p - 1} \sigma
\quad \text{in} \ 2B.
\] 
Then for any $\beta \geq 1$,
$v^{\beta} \in \dot{W}_{0}^{1, p}(2B) \cap L^{\infty}(2B)$.
Moreover,
\[
\int_{2B}
|\nabla v^{\beta}|^{p - 2} \nabla v^{\beta} \cdot \nabla \varphi \, dx
\leq
\beta^{p - 1}
\int_{2B}
|\nabla v|^{p - 2} \nabla v \cdot \nabla (v^{(\beta - 1)(p - 1)} \varphi) \, dx
\]
for any nonnegative $\varphi \in C_{c}^{\infty}(2B)$.
In other words,
\begin{equation*}\label{weak_form_of_subsol}
- \laplacian_{p} v^{\beta}
\leq
\beta^{p - 1} v^{(\beta - 1)(p - 1)} \tilde{\sigma}
\quad \text{in} \ 2B
\end{equation*}
in the sense of distribution.
Let
\begin{equation}\label{definition_of_u_0}
u_{0} = c_{1} v^{\frac{p - 1}{p - 1 - q}},
\end{equation}
where
$c_{1} := \min\{ c_{0} c_{K}^{\frac{(1 - p)}{p - 1 - q}}, 1 \}$.
Here, $c_{K}$ and $c_{0}$ are the constants in
Theorems \ref{pointwise_est_p-superharmonic}
and \ref{thm:lower_bound_of_sol_to_subeq},
respectively.
Then $u_{0} \in \dot{W}_{0}^{1, p}(2B) \cap L^{\infty}(2B)$ and
\[
- \laplacian_{p} u_{0}
\leq (c_{1})^{p - 1 - q} \sigma u_{0}^{q}
\leq \sigma u_{0}^{q}
\quad \text{in} \ 2B.
\]
Moreover, by Theorem \ref{pointwise_est_p-superharmonic},
\begin{equation}\label{upper_bound_of_u0}
u_{0} \leq c_{0} (\W_{1, p} \sigma)^{ \frac{p - 1}{p - 1 - q} }
\quad \text{in $2B$.}
\end{equation}
We define a sequence of $p$-superharmonic functions $\{  u_{j} \}_{j=1}^{\infty}$ by
\begin{equation}\label{definition_of_u_j}
- \laplacian_{p} u_{j+1} = \sigma u_{j}^{q} + \mu, \quad \text{in} \ 2B, \ j = 0,1,2, \dots.
\end{equation}
Assume that $u_{j}$ is bounded for some $j \geq 0$.
Then the measure $\sigma u_{j}^{q} + \mu$ belongs to the dual of $\dot{W}_{0}^{1, p}(2B)$,
and $\W_{1, p}( u_{j}^{q} d \sigma + d \mu )$ is bounded.
By the comparison principle for weak solutions, $u_{j + 1} \geq u_{j}$ for all $j \geq 0$.
Therefore, $\{  u_{j} \}_{j=1}^{\infty}$ is defined as
an increasing sequence of bounded finite energy $p$-superharmonic functions.
By Theorem \ref{pointwise_est_p-superharmonic},
for every $x \in B$,
\begin{equation}\label{lower_bound_of_u1}
\begin{split}
u_{1}(x)
& \geq
\max\left\{
u_{0}(x), \frac{1}{c_{K}} \W_{1, p}^{\frac{R}{2}} \mu(x)
\right\} \\
& \geq
\frac{1}{C}
\left\{
(\W_{1, p}^{\frac{R}{2}} \sigma)^{ \frac{p - 1}{p - 1 - q} }(x) + \W_{1, p}^{\frac{R}{2}} \mu(x)
\right\}.
\end{split}
\end{equation}
Assume that $w$ is a nonnegative $p$-superharmonic supersolution to \eqref{eq:p-laplacian}.
Then by Theorem \ref{thm:lower_bound_of_sol_to_subeq} and \eqref{upper_bound_of_u0},
\[
u_{0} \leq c_{0} (\W_{1, p} \sigma)^{ \frac{p - 1}{p - 1 - q} } \leq w
\quad \text{$d \sigma$-a.e. in $2B$.}
\]
Since $\sigma$ is absolutely continuous with respect to the $p$-capacity,
it follows from
the comparison principle for renormalized solutions
(see \cite[Lemma 5.2]{MR3567503}) that
$w \geq u_{1}$ in $2B$.
Thus, by induction,
\begin{equation}\label{upper_bound_of_uj}
w \geq u_{j} \quad \text{in} \ 2B
\end{equation}
for all $j \geq 1$.
The same argument is valid if $w$ satisfies
assumptions in \ref{comparison_principle_for_sols_to_lp}.

\noindent
\textbf{Step 2.}
We give bounds of $\{ u_{j} \}_{j = 1}^{\infty}$.
For simplicity, we denote by $u_{j}$ the zero extension of $u_{j}$ again.
By \eqref{upper_bound_of_u0} and \eqref{cond:dsigma-wolff},
$u_{0} \in L^{\gamma + q}(2B, d \sigma)$.
Assume that $u_{j} \in L^{\gamma + q}(2B, d \sigma)$ for some $j \geq 0$.
Then by Theorem \ref{pointwise_est_p-superharmonic},
the comparison principle and \eqref{quasi-linearity},
\[
\begin{split}
\| u_{j + 1} \|_{L^{\gamma + q}(2B, d \sigma)} 
& \leq
c_{K} \| \W_{1, p} ( u_{j}^{q} d \sigma + d \mu) \|_{L^{\gamma + q}(2B, d \sigma)} \\
& \leq
C \| \W_{1, p} ( u_{j}^{q} d \sigma) \|_{L^{\gamma + q}(2B, d \sigma)}
+
C \| \W_{1, p}\mu \|_{L^{\gamma + q}(2B, d \sigma)} \\
& =
C \| \W_{1, p} ( u_{j}^{q} d \sigma) \|_{L^{\gamma + q}(\R^{n}, d \sigma)}
+
C \| \W_{1, p}\mu \|_{L^{\gamma + q}(\R^{n}, d \sigma)}.
\end{split}
\]
By Lemma \ref{lem:weighted_norm_inequality},
\begin{equation}\label{sub-natural_estimate}
\begin{split}
\| \W_{1, p} ( u_{j}^{q} d \sigma) \|_{L^{\gamma + q}(\R^{n}, d \sigma)}
\leq
C \| u_{j}^{q} \|_{L^{\frac{\gamma + q}{q}}(\R^{n}, d \sigma)}^{\frac{1}{p - 1}}
=
C \| u_{j} \|_{L^{\gamma + q}(2B, d \sigma)}^{\frac{q}{p - 1}}.
\end{split}
\end{equation}
Moreover, by Theorem \ref{thm:mutual_energy_estimate} and \eqref{cond:dmu-wolff},
\[
\| \W_{1, p}\mu \|_{L^{\gamma + q}(\R^{n}, d \sigma)}
\leq
C.
\]
Thus,
\[
\| u_{j + 1} \|_{L^{\gamma + q}(2B, d \sigma)}
\leq
C( \| u_{j} \|_{L^{\gamma + q}(2B, d \sigma)}^{\frac{q}{p - 1}} + 1).
\]
Since $q < p - 1$, by Young's inequality and monotonicity of $u_{j}$,
\begin{equation}\label{integral_bound_of_u}
\| u_{j + 1} \|_{L^{\gamma + q}(2B, d \sigma)} \leq C.
\end{equation}
Then by H\"{o}lder's inequality and
Lemma \ref{lem:weighted_norm_inequality}, 
\begin{equation*}
\begin{split}
\left(
\int_{\R^{n}}
(\W_{1, p}(u_{j}^{q} d \sigma))^{\gamma}
u_{j}^{q} \, d \sigma
\right)^{\frac{1}{p - 1 + \gamma}}
& \leq
C \| u_{j}^{q} \|_{L^{\frac{\gamma + q}{q}}(\R^{n}, d \sigma)}^{\frac{1}{p - 1}}
=
C \| u_{j} \|_{L^{\gamma + q}(2B, d \sigma)}^{\frac{q}{p - 1}}.
\end{split}
\end{equation*}
Hence, by \eqref{integral_bound_of_u}, 
\begin{equation}\label{energy_bound_of_u}
\begin{split}
\int_{\R^{n}}
(\W_{1, p}(u_{j}^{q} d \sigma))^{\gamma}
u_{j}^{q} \, d \sigma
\leq
C.
\end{split}
\end{equation}
Fix $l > 0$.
Testing \eqref{definition_of_u_j} with $\min\{ u_{j + 1}, l \}$, we get
\[
\begin{split}
\int_{2B} |\nabla \min\{ u_{j + 1}, l \}|^{p} \, dx
& \leq
l
\left(
\int_{2B} u_{j}^{q} \, d \sigma
+
\mu(2B)
\right) \\
& \leq
l \left(
\| u_{j} \|_{L^{\gamma + q}(2B, d \sigma)}^{q}
\sigma(2B)^{\frac{\gamma}{\gamma + q}}
+
\mu(2B)
\right).
\end{split}
\]
By using \eqref{integral_bound_of_u} again,
\begin{equation}\label{truncated_energy_bound_of_u}
\int_{2B} |\nabla \min\{ u_{j + 1}, l \}|^{p} \, dx
\leq
C l.
\end{equation}
Here, 
the constant $C$ depends also on $\sigma(2B)$ and $\mu(2B)$, but not on $j \in \N$.

\noindent 
\textbf{Step 3.}
Let
\[
u = \lim_{j \to \infty} u_{j}.
\]
By \eqref{truncated_energy_bound_of_u} and the Poincar\'{e} inequality,
$u$ is not identically infinite.
Hence,
$u$ is $p$-superharmonic in $2B$ 
by \cite[Lemma 7.3]{MR2305115}.
Moreover, by the weakly continuity result in
\cite[Theorem 3.1]{MR1890997},
we have $\nu[u] = \sigma u^{q} + \mu$.
Applying the monotone convergence theorem, we have
\begin{equation}\label{localized_energy}
\int_{\R^{n}}
(\W_{1, p}(u^{q} d \sigma))^{\gamma}
u^{q} \, d \sigma
\leq
C.
\end{equation}
By Corollary \ref{cor:quasi-linearity_of_energy_space},
\begin{equation*}
\int_{\R^{n}}
(\W_{1, p} \nu[u])^{\gamma}
d \nu[u]
\approx
\int_{\R^{n}}
(\W_{1, p}(u^{q} d \sigma))^{\gamma}
u^{q} \, d \sigma
+
\int_{\R^{n}}
(\W_{1, p} \mu)^{\gamma}
\, d \mu.
\end{equation*}
Thus, by Theorem \ref{pointwise_est_p-superharmonic},
Corollary \ref{cor:lorentz_integrability_of_wolff_potentials},
\eqref{localized_energy} and \eqref{cond:dmu-wolff},
\[
\| u \|_{L^{r, \rho}(2B)}
\leq
c_{K} \| \W_{1, p} \nu[u] \|_{L^{r, \rho}(\R^{n})}
\leq
C.
\]
Also, by \eqref{truncated_energy_bound_of_u},
$\min \{ u, l \} \in \dot{W}_{0}^{1, p}(2B)$ for all $l > 0$.
The properties
\ref{lower_bound_of_sols_to_lp}, \ref{upper_bound_of_sols_to_lp}
and \ref{comparison_principle_for_sols_to_lp}
follow from \eqref{lower_bound_of_u1} and \eqref{upper_bound_of_uj}.
\end{proof}

\begin{proof}[Proof of Theorem \ref{thm:main-p-laplacian}]
For each $k \in \N$, we consider measures
$\sigma_{k} = \mathbf{1}_{ \Omega(\sigma, k) } \sigma$
and $\mu_{k} = \mathbf{1}_{ \Omega(\mu, k) } \mu$,
where $\Omega(\sigma, k)$ and $\Omega(\mu, k)$ are defined by \eqref{cutoff_set}.
Using Lemma \ref{lem:main-p-laplacian},
we construct a sequence of $p$-superharmonic functions $\{ u_{k} \}_{k = 1}^{\infty}$ satisfying
\begin{equation}\label{eq:p-laplace_Bj}
- \laplacian_{p} u_{k} = \sigma_{k} u_{k}^{q} + \mu_{k} \quad \text{in} \ B(0, 2^{k + 1})
\end{equation}
and
\begin{equation*}
\min\{ u, l \} \in W_{0}^{1, p}(B(0, 2^{k + 1})) \quad \forall l > 0.
\end{equation*}
By the comparison principle for weak solutions and
\ref{comparison_principle_for_sols_to_lp} in Lemma \ref{lem:main-p-laplacian},
\[
u_{k + 1} \geq u_{k} \quad \text{in} \ B(0, 2^{k + 1}).
\]
We denote by $u_{k}$ the zero extension of $u_{k}$ again.

Let $u = \lim_{k \to \infty} u_{k}$.
Then by the monotone convergence theorem,
\begin{equation}\label{global_lorentz}
\| u \|_{L^{r, \rho}(\R^{n})} = \lim_{k \to \infty} \| u_{k} \|_{L^{r, \rho}(\R^{n})} \leq C.
\end{equation}
By Remark \ref{rem:approximation_of_wolff_potential},
$u_{k}^{q} \mathbf{1}_{\Omega(\sigma, k)} \uparrow u^{q}$ $d \sigma$-a.e. 
Hence, using the monotone convergence theorem twice, we get
\begin{equation}\label{global_energy}
\int_{\R^{n}}
(\W_{1, p}(u^{q} d \sigma))^{\gamma}
u^{q} \, d \sigma
\leq
C.
\end{equation}
By \eqref{global_lorentz},
$u$ is not identically infinite, and $\liminf_{|x| \to \infty} u(x) = 0$.
Therefore, $u$ is $p$-superharmonic in $\R^{n}$ and $\nu[u] = \sigma u^{q} + \mu$.
By Theorem \ref{pointwise_est_p-superharmonic}
and Corollary \ref{cor:quasi-linearity_of_energy_space},
\begin{equation*}
\begin{split}
\int_{\R^{n}} u^{\gamma} \, d \nu[u]
& \approx
\int_{\R^{n}} (\W_{1, p} \nu[u])^{\gamma} \, d \nu[u] \\
& \approx
\int_{\R^{n}} (\W_{1, p}(u^{q} d \sigma))^{\gamma} u^{q} \, d \sigma
+
\int_{\R^{n}} (\W_{1, p} \mu)^{\gamma} \, d \mu.
\end{split}
\end{equation*}
Thus, \eqref{generalized_energy} is finite
by \eqref{global_energy} and \eqref{cond:dmu-wolff}.

Fix any $x \in \R^{n}$.
Then by \ref{lower_bound_of_sols_to_lp} in Lemma \ref{lem:main-p-laplacian},
\[
\begin{split}
u(x)
\geq
u_{k}(x)
\geq
\frac{1}{C}
\left\{
( \W_{1, p}^{2^{k - 1}} \sigma_{k})^{ \frac{p - 1}{p - 1 - q} }(x) + (\W_{1, p}^{2^{k - 1}} \mu_{k})(x)
\right\}
\end{split}
\]
whenever $x \in B(0, 2^{k})$.
Passing to the limit $k \to \infty$ and
applying the monotone convergence theorem to the right-hand side, we get
\[
\begin{split}
u(x)
\geq
\frac{1}{C}
\left\{
(\W_{1, p} \sigma)^{ \frac{p - 1}{p - 1 - q} }(x) + (\W_{1, p} \mu)(x)
\right\}.
\end{split}
\]
Hence, $u$ is positive in $\R^{n}$.
If $w$ is a $p$-superharmonic supersolution to \eqref{eq:p-laplacian}, then
$w \geq u_{k}$ for all $k \in \N$
by \ref{upper_bound_of_sols_to_lp} in Lemma \ref{lem:main-p-laplacian}.
Therefore, $w \geq u$.
This implies that $u$ is minimal.
\end{proof}

\begin{Cor}\label{cor:main-p-laplacian}
Assume
that there exists a positive $p$-superharmonic supersolution
$w \in L^{q}_{\loc}(\R^{n}, d \sigma)$ to \eqref{eq:p-laplacian}
satisfying \eqref{generalized_energy} for some $\gamma > 0$.
Then there exists a minimal $p$-superharmonic solution
$u \in L^{r, \rho}(\R^{n})$ to \eqref{eq:p-laplacian}
satisfying \eqref{generalized_energy} and $w \geq u$ in $\R^{n}$.
\end{Cor}

\begin{proof}
In this case, \eqref{cond:dsigma-wolff} and \eqref{cond:dmu-wolff} are fulfilled
by Theorems
\ref{thm:lower_bound_of_sol_to_subeq} and \ref{pointwise_est_p-superharmonic},
respectively.
Hence 
there exists a minimal $p$-superharmonic solution $u$ to \eqref{eq:p-laplacian}
by Theorem \ref{thm:main-p-laplacian}.
By minimality of $u$, $w \geq u$ in $\R^{n}$.
\end{proof}

\begin{Cor}\label{cor:caccioppoli}
Assume that the assumptions of Theorem \ref{thm:main-p-laplacian}
hold with $\gamma \geq 1$.
Let $u$ be the minimal solution in Theorem \ref{thm:main-p-laplacian}.
Then $u$ belongs to $W^{1, p}_{\loc}(\R^{n})$
and satisfies \eqref{eq:p-laplacian} in the sense of weak solutions.
\end{Cor}

\begin{proof}
We give additional estimates for 
solutions $\{ u_{j} \}_{j = 1}^{\infty}$ in Lemma \ref{lem:main-p-laplacian}.
By Corollary \ref{cor:quasi-linearity_of_energy_space},
\begin{equation*}
\int_{\R^{n}}
(\W_{1, p} \nu[u_{j}])^{\gamma}
d \nu[u_{j}]
\approx
\int_{\R^{n}}
(\W_{1, p}(u_{j}^{q} d \sigma))^{\gamma}
u_{j}^{q} \, d \sigma
+
\int_{\R^{n}}
(\W_{1, p} \mu)^{\gamma}
\, d \mu.
\end{equation*}
The right-hand side is estimated by
\ref{norm_bounds_for_sols_to_lp} in Lemma \ref{lem:main-p-laplacian}
and \eqref{cond:dmu-wolff}.
Thus, by Theorem \ref{pointwise_est_p-superharmonic},
\[
\begin{split}
\| u_{j + 1} \|_{L^{\gamma}(2B, d \mu)}
& \leq
c_{K} \| \W_{1, p} \nu[u_{j}] \|_{L^{\gamma}(2B, d \mu)} \\
& \leq
c_{K} \| \W_{1, p} \nu[u_{j}] \|_{L^{\gamma}(\R^{n}, d \nu[u_{j}])}
\leq
C.
\end{split}
\]
Fix a ball $B(x_{0}, R_{0}) \subset B$, and
take a nonnegative function $\eta \in C_{c}^{\infty}(B(x_{0}, 2R_{0}))$ satisfying
$\eta \equiv 1$ on $B(x_{0}, R_{0})$
and $|\nabla \eta| \leq C / R_{0}$.
Testing \eqref{definition_of_u_j} with $u_{j + 1} \eta^{p}$
and using Young's inequality, we get
\[
\begin{split}
&
\int_{2B} |\nabla u_{j + 1} |^{p} \eta^{p} \, dx \\
& \leq
C \left(
\int_{2B} u_{j + 1}^{p} |\nabla \eta|^{p} \, dx
+
\int_{2B} u_{j + 1} \eta^{p} u_{j}^{q} \, d \sigma
+
\int_{2B} u_{j + 1} \eta^{p} \, d \mu
\right) \\
& \leq
C \left(
\frac{C}{R_{0}^{p}}
\int_{B(x_{0}, 2R_{0})} u_{j + 1}^{p} \, dx
+
\int_{B(x_{0}, 2R_{0})} u_{j + 1}^{1 + q} \, d \sigma
+
\int_{B(x_{0}, 2R_{0})} u_{j + 1} \, d \mu
\right).
\end{split}
\]
Since $\gamma \geq 1$,
by
\ref{norm_bounds_for_sols_to_lp} in Lemma \ref{lem:main-p-laplacian}
and H\"{o}lder's inequality,
\[
\| u_{j + 1} \|_{L^{p}(B(x_{0}, 2R_{0}))}
\leq
C R_{0}^{\frac{n}{p} - \frac{n - p}{p - 1 + \gamma}}
\| u \|_{L^{r, \rho}(2B)}
\leq
C
\]
and
\[
\| u_{j + 1} \|_{W^{1, p}(B(x_{0}, R_{0}))} \leq C'.
\]
Here, the constant $C'$ depends also on
$R_{0}$, $\sigma( \cl{B(x_{0}, 2R_{0})} )$ and $\mu( \cl{B(x_{0}, 2R_{0})} )$,
but not on $j \in \N$ and $B$.
Therefore, $\lim_{j \to \infty} u_{j} = u \in W^{1, p}(B(x_{0}, R_{0}))$
and $\| u \|_{W^{1, p}(B(x_{0}, R_{0}))} \leq C'$.
Applying the same limit argument to
$\{ u_{k} \}_{k = 1}^{\infty}$ in Theorem \ref{thm:main-p-laplacian},
we see that $\lim_{k \to \infty} u_{k}= u \in W^{1, p}_{\loc}(\R^{n})$.
\end{proof}

\section{Remarks for the endpoint cases}\label{sec:endpoint}

When $\gamma = \infty$, we replace \eqref{cond:dsigma-wolff} and \eqref{cond:dmu-wolff} 
with the following conditions:
\begin{equation}\tag{\ref{cond:dsigma-wolff}$'$}\label{cond:dsigma-wolff_infty}
\| \W_{1, p} \sigma \|_{L^{\infty}(\R^{n}, d \sigma)}
< \infty,
\end{equation}
\begin{equation}\tag{\ref{cond:dmu-wolff}$'$}\label{cond:dmu-wolff_infty}
\| \W_{1, p} \mu \|_{L^{\infty}(\R^{n}, d \mu)}
< \infty.
\end{equation}

\begin{Thm}\label{thm:main-p-laplacian_infty}
Let $1 < p < n$ and $0 < q < p - 1$.
Let $\sigma, \mu \in \M^{+}(\R^{n})$ with $(\sigma, \mu) \not\equiv (0, 0)$.
Assume that \eqref{cond:dsigma-wolff_infty} and \eqref{cond:dmu-wolff_infty} hold.
Then there exists a minimal bounded weak solution
$u \in W^{1, p}_{\loc}(\R^{n}) \cap L^{\infty}(\R^{n})$ to \eqref{eq:p-laplacian}.
Conversely,
if there exists a bounded $p$-superharmonic supersolution
$u \in L^{q}_{\loc}(\R^{n}, d \sigma)$ to \eqref{eq:p-laplacian},
then \eqref{cond:dsigma-wolff_infty} and \eqref{cond:dmu-wolff_infty} are fulfilled.
\end{Thm}

\begin{proof}
We consider solutions to \eqref{eq:p-laplace_B}.
As in the proof of Lemma \ref{lem:main-p-laplacian},
let $\{ u_{j} \}_{j = 1}^{\infty}$
be the sequence of $p$-superharmonic functions defined in
\eqref{definition_of_u_0} and \eqref{definition_of_u_j}.
Then by \eqref{cond:dsigma-wolff_infty}
and \eqref{wolff-max},
\begin{equation}\label{sub-natural_estimate_infty}
\begin{split}
\| \W_{1, p} ( u_{j}^{q} d \sigma) \|_{L^{\infty}(2B)}
\leq
C \| u_{j} \|_{L^{\infty}(2B, d \sigma)}^{\frac{q}{p - 1}}
\end{split}
\end{equation}
for all $j \geq 0$.
Replacing \eqref{sub-natural_estimate} by \eqref{sub-natural_estimate_infty}
and using \eqref{cond:dmu-wolff_infty}, we get
\[
\| u_{j + 1} \|_{L^{\infty}(2B)}
\leq
C\left( \| u_{j} \|_{L^{\infty}(2B)}^{\frac{q}{p - 1}} + 1 \right).
\]
Therefore,
\[
\| u_{j} \|_{L^{\infty}(2B)} \leq C
\]
for all $j \geq 1$
and $u = \lim_{j \to \infty} u_{j}$ is a bounded solution to \eqref{eq:p-laplace_B}.
Using solutions to \eqref{eq:p-laplace_Bj},
we construct a $p$-superharmonic function $u$ in $\R^{n}$
satisfying $\nu[u] = \sigma u^{q} + \mu$ in $\R^{n}$.
Then by the bound of solutions to \eqref{eq:p-laplace_Bj},
$u$ is bounded on $\R^{n}$.
Moreover, by Theorem \ref{pointwise_est_p-superharmonic},
\[
\begin{split}
u_{k}(x)
\leq
c_{K} \W_{1, p} \nu[u_{k}](x)
\leq
C
\| u \|_{L^{\infty}(\R^{n})}^{\frac{q}{p - 1}}
\W_{1, p} \sigma(x)
+
C  \W_{1, p} \mu(x)
\end{split}
\]
for all $k \geq 1$ and for all $x \in \R^{n}$.
By \cite[Corollary 3.2]{MR3567503},
this implies that $\liminf_{|x| \to \infty} u(x) = 0$.
From the arguments in the proof of Theorem \ref{thm:main-p-laplacian},
it follows that $u$ is a minimal $p$-superharmonic solution to \eqref{eq:p-laplacian}.
Since $u$ is a bounded $p$-superharmonic function,
$u \in W^{1, p}_{\loc}(\R^{n})$
by  \cite[Theorem 7.25]{MR2305115}.

The converse part follows from
Theorems \ref{pointwise_est_p-superharmonic} 
and \ref{thm:lower_bound_of_sol_to_subeq}.
\end{proof}

The case of $\gamma = 0$ is more delicate.
We give a sufficient condition for the existence of minimal solutions to \eqref{eq:p-laplacian}.

\begin{Prop}\label{thm:main-p-laplacian_zero}
Let $1 < p < n$ and $0 < q < p - 1$.
Let $\sigma, \mu \in \M^{+}(\R^{n})$ with $(\sigma, \mu) \not\equiv (0, 0)$.
Assume that there exists a positive function $w \in L^{q}(\R^{n}, d \sigma)$ satisfying
\[
w =  \W_{1, p}(w^{q} d \sigma) \quad \text{in} \ \R^{n}.
\]
Assume also that $\mu$ is finite and absolutely continuous with respect to the $p$-capacity.
Then there exists a minimal $p$-superharmonic solution
$u \in L^{q}(\R^{n}, d \sigma) \cap L^{\frac{n(p - 1)}{n - p}, \infty}(\R^{n})$
to \eqref{eq:p-laplacian}.
Moreover, the Riesz measure of $u$ is finite.
\end{Prop}

\begin{proof}
As above, we consider the localized problem \eqref{eq:p-laplace_B}
and its approximate solutions $\{ u_{j} \}_{j = 1}^{\infty}$.
Then by \cite[Theorem 4.4]{MR3567503}, we have
\[
\| \W_{1, p} ( u_{j}^{q} d \sigma) \|_{L^{q}(2B, d \sigma)}
\leq
C \| u_{j} \|_{L^{q}(2B, d \sigma)}^{\frac{q}{p - 1}}
\]
and
\[
\| \W_{1, p} \mu \|_{L^{q}(2B, d \sigma)}
\leq
C.
\]
From Theorem \ref{pointwise_est_p-superharmonic} and \eqref{quasi-linearity},
it follows that
\[
\| u_{j + 1} \|_{L^{q}(2B, d \sigma)}
\leq
C( \| u_{j} \|_{L^{q}(2B, d \sigma)}^{\frac{q}{p - 1}} + 1).
\]
Thus, 
$u = \lim_{j \to \infty} u_{j}$ belongs to $L^{q}(2B, d \sigma)$,
and its Riesz measure $\nu[u]$ satisfies
\[
\left(
\int_{2B} \, d \nu[u]
\right)^{\frac{1}{p - 1}}
=
\left(
\int_{2B} u^{q} \, d \sigma
+
\int_{2B} \, d \mu
\right)^{\frac{1}{p - 1}}
\leq
C.
\]
Since $\min\{ u, l \} \in W_{0}^{1, p}(2B)$ for all $l > 0$,
it follows from \cite[Lemma 4.1]{MR1354907} that
\begin{equation}\label{weak_lp_bound}
\| u \|_{L^{\frac{n(p - 1)}{n - p}, \infty}(2B)}
\leq C.
\end{equation}
Using solutions to \eqref{eq:p-laplace_Bj},
we construct a $p$-superharmonic function $u$ in $\R^{n}$ satisfying
$\nu[u] = \sigma u^{q} + \mu$.
Then by \eqref{weak_lp_bound} and
the Fatou property of the $L^{r, \infty}$ norm (see, e.g., \cite[p.14]{MR2445437}),
\[
\| u \|_{L^{\frac{n(p - 1)}{n - p}, \infty}(\R^{n})} \leq C.
\]
Hence $\liminf_{|x| \to \infty} u(x) = 0$.
Using the same argument as in the proof of Theorem \ref{thm:main-p-laplacian},
we see that $u$ is a minimal $p$-superharmonic solution to \eqref{eq:p-laplacian}.
By the monotone convergence theorem, $u \in L^{q}(\R^{n}, d \sigma)$.
Hence, the Riesz measure of $u$ is finite.
\end{proof}

\section{Elliptic equations with several sub-natural growth terms}\label{sec:several_terms}

\begin{Thm}\label{thm:main-p-laplacian_several_terms}
Let $1 < p < n$.
Let $\{ q_{m} \}_{m = 1}^{M}$ be positive constants satisfying
\begin{equation}\label{cond_q_several}
0 < q_{m} < p - 1,
\quad m = 1, \dots, M.
\end{equation}
Let $\{ \sigma^{(m)} \}_{m = 1}^{M}$ and $\mu$ be Radon measures on $\R^{n}$
such that $(\sigma^{(1)}, \dots, \sigma^{(m)}, \mu)  \neq (0, \dots, 0, 0)$.
Assume that there exists a positive constant $0 < \gamma < \infty$ such that
\begin{equation}\label{cond_sigma_several}
\int_{\R^{n}}
(\W_{1, p} \sigma^{(m)})^{\frac{(\gamma + q_{m})(p - 1)}{p - 1 - q_{m}}}
\, d \sigma^{(m)}
< \infty,
\quad
m = 1, \dots, M,
\end{equation}
\begin{equation}\label{cond_mu_several}
\int_{\R^{n}} (\W_{1, p} \mu)^{\gamma} \, d \mu
< \infty.
\end{equation}
Then there exists a minimal $p$-superharmonic solution $u$ to \eqref{eq:several_terms}.
Moreover, $u$ satisfies \eqref{generalized_energy}
and belongs to $L^{r, \rho}(\R^{n})$,
where $r = n(p - 1 + \gamma) / (n - p)$ and $\rho = p - 1 + \gamma$.
\end{Thm}

\begin{proof}
First, we consider localized problems.
For $k \in \N$, we consider measures
$\sigma^{(m)}_{k} = \mathbf{1}_{ \Omega(\sigma^{(m)}, k) } \sigma$
and $\mu_{j} = \mathbf{1}_{ \Omega(\mu, j) } \mu$,
where $\Omega(\sigma^{(m)}, k)$ and $\Omega(\mu, k)$ are defined by \eqref{cutoff_set}.
For simplicity of notation,
we temporarily write $B(0, 2^{k})$, $\sigma^{(m)}_{k}$ and $\mu_{k}$
as $B$, $\sigma^{(m)}$ and $\mu$ respectively.

Next, we construct a $p$-superharmonic function $u$ satisfying
\begin{equation}\label{eq:several_terms_localized}
- \laplacian_{p} u = \sum_{m = 1}^{M} \sigma^{(m)} u^{q_{m}} + \mu \quad \text{in} \ 2B 
\end{equation}
and
\begin{equation*}
\min\{ u, l \} \in \dot{W}_{0}^{1, p}(2B) \quad \forall l > 0.
\end{equation*}
Let $\{ v^{(m)} \} \subset \dot{W}_{0}^{1, p}(2B) \cap L^{\infty}(2B)$ be the weak solutions to
\[
- \laplacian_{p} v^{(m)} = \left( \frac{p - 1 - q_{m}}{p - 1} \right)^{p - 1} \sigma^{(m)}
\quad \text{in} \ 2B.
\]
For $m = 1, \dots M$,
put $u_{0}^{(m)} = c_{1}^{(m) }(v^{(m)})^{\frac{p - 1}{p - 1 - q_{m}}}$,
where $\{ c_{1}^{(m)} \}_{m = 1}^{M} \subset (0, 1]$ are constants
depending only on $n$, $p$ and $q$.
Then $u_{0}^{(m)}$ satisfy
\[
- \laplacian_{p} u_{0}^{(m)}
\leq \sigma_{k}^{(m)} (u_{0}^{(m)})^{q_{m}}
\quad \text{in} \ 2B.
\]
As in the proof of Lemma \ref{lem:main-p-laplacian},
we can choose small $\{ c_{1}^{(m)} \}_{m = 1}^{M}$ such that
\begin{equation}\label{upper_bound_of_uj_several}
w \geq u_{0}^{(m)} \quad \text{in $2B$} \quad (m = 1, \ldots M)
\end{equation}
whenever 
$w$ is a nonnegative $p$-superharmonic supersolution to \eqref{eq:several_terms}.
We define an increasing sequence of $p$-superharmonic functions
$\{  u_{j} \}_{j=1}^{\infty} \subset \dot{W}_{0}^{1, p}(2B)$ by
\[
- \laplacian_{p} u_{1}
= \sum_{m = 1}^{M} \sigma^{(m)} (u_{0}^{(m)})^{q_{m}} + \mu
\]
and
\begin{equation}\label{eq:def_of_uj@several_terms_localized}
- \laplacian_{p} u_{j+1}
= \sum_{m = 1}^{M} \sigma^{(m)} u_{j}^{q_{m}} + \mu, \quad  j=1,2, \dots.
\end{equation}
By Theorem \ref{pointwise_est_p-superharmonic} and \eqref{quasi-linearity},
for each $m$,
\[
\begin{split}
&
\| u_{j + 1} \|_{L^{\gamma + q_{m}}(2B, d \sigma^{(m)})} \\
& \leq
C \sum_{l = 1}^{M}
\| \W_{1, p}(u_{j}^{q_{l}} d \sigma^{(l)}) \|_{ L^{\gamma + q_{m}}(\R^{n}, d \sigma^{(m)}) }
+
C \| \W_{1, p} \mu \|_{L^{\gamma + q_{m}}(\R^{n}, d \sigma^{(m)})}.
\end{split}
\]
Since $\sigma^{(m)}$ satisfies \eqref{cond_sigma_several},
Theorem \ref{thm:mutual_energy_estimate} yields
\[
\begin{split}
\| \W_{1, p}(u_{j}^{q_{l}} d \sigma^{(l)}) \|_{L^{\gamma + q_{m}}(\R^{n}, d \sigma^{(m)})}
& =
\left(
\int_{\R^{n}} (\W_{1, p}(u_{j}^{q_{l}} d \sigma^{(l)}))^{\gamma + q_{m}} \, d \sigma^{(m)}
\right)^{\frac{1}{ \gamma + q_{m} } }
 \\
& \leq
C
\left(
\int_{\R^{n}} (\W_{1, p}(u_{j}^{q_{l}} d \sigma^{(l)}))^{\gamma} u_{j}^{q_{l}} \, d \sigma^{(l)}
\right)^{ \frac{1}{p - 1 + \gamma} }.
\end{split}
\]
Moreover, 
since $\sigma^{(l)}$ satisfies \eqref{cond_sigma_several},
by H\"{o}lder's inequality and Lemma \ref{lem:weighted_norm_inequality},
\begin{equation}\label{energy_several_terms}
\begin{split}
\left(
\int_{\R^{n}} (\W_{1, p}(u_{j}^{q_{l}} d \sigma^{(l)}))^{\gamma} u_{j}^{q_{l}} \, d \sigma^{(l)}
\right)^{ \frac{1}{p - 1 + \gamma} }
& \leq
C \| u_{j}^{q_{l}} \|_{L^{\frac{\gamma + q_{l}}{q_{l}}}(\R^{n}, d \sigma^{(l)})}^{\frac{1}{p - 1}}
\\
& =
C \| u_{j} \|_{L^{\gamma + q_{l}}(2B, d \sigma^{(l)})}^{\frac{q_{l}}{p - 1}}.
\end{split}
\end{equation}
By \eqref{cond_sigma_several}, \eqref{cond_mu_several}
and Theorem \ref{thm:mutual_energy_estimate},
\[
\| \W_{1, p} \mu \|_{L^{\gamma + q_{m}}(\R^{n}, d \sigma^{(m)})}
\leq
C.
\]
Therefore,
\[
\sum_{m = 1}^{M} \| u_{j + 1} \|_{L^{\gamma + q_{m}}(2B, d \sigma^{(m)})}
\leq
C \sum_{m = 1}^{M} \| u_{j} \|_{L^{\gamma + q_{m}}(2B, d \sigma^{(m)})}^{\frac{q_{m}}{p - 1}}
+
C.
\]
By \eqref{cond_q_several} and Young's inequality,
this implies that
\[
\sum_{m = 1}^{M} \| u_{j + 1} \|_{L^{\gamma + q_{m}}(2B, d \sigma^{(m)})}
\leq
C.
\]
Using \eqref{energy_several_terms} again,
we get a uniform bound corresponding to \eqref{energy_bound_of_u}.
Let $u = \lim_{j \to \infty} u$.
Then $u$ satisfies \eqref{eq:several_terms_localized}.
Moreover, by Corollary \ref{cor:lorentz_integrability_of_wolff_potentials},
\[
\begin{split}
\sum_{m = 1}^{M}
\left(
\int_{\R^{n}}
(\W_{1, p}(u^{q_{m}} d \sigma^{(m)} ))^{\gamma}
u^{q_{m}} \, d \sigma^{(m)}
\right)^{\frac{1}{p - 1 + \gamma}}
+
\| u \|_{L^{r, \rho}(2B)}
\leq
C.
\end{split}
\]
By Theorem \ref{pointwise_est_p-superharmonic},
for any $x \in B$,
\begin{equation}\label{lower_bound_several_terms}
\begin{split}
u(x)
& \geq
\max
\left\{
u_{0}^{(1)}(x), \dots, u_{0}^{(M)}(x),
\frac{1}{c_{K}} \W_{1, p}^{\frac{R}{2}} \mu(x)
\right\} \\
& \geq
\frac{1}{C}
\left\{
\sum_{m = 1}^{M} (\W_{1, p}^{\frac{R}{2}} \sigma^{(m)})^{ \frac{p - 1}{p - 1 - q} }(x)
+ \W_{1, p}^{\frac{R}{2}} \mu(x)
\right\}.
\end{split}
\end{equation}
If $w$ is a nonnegative $p$-superharmonic supersolution to \eqref{eq:several_terms},
then by \eqref{upper_bound_of_uj_several} and induction,
\begin{equation}\label{upper_bound_several_terms}
w \geq u \quad \text{in} \ 2B.
\end{equation}

Finally, we construct a minimal $p$-superharmonic solution to \eqref{eq:several_terms}.
For each $k \geq 1$, let $u_{k}$ be a $p$-superharmonic function
satisfying \eqref{eq:several_terms_localized}.
Passing to the limit $k \to \infty$, we get a $p$-superharmonic function
$u \in L^{r, \rho}(\R^{n})$ satisfying
$\nu[u] = \sum_{m = 1}^{M} \sigma^{(m)} u^{q_{m}} + \mu$.
By Theorem \ref{pointwise_est_p-superharmonic} and
Corollary \ref{cor:quasi-linearity_of_energy_space},
\begin{equation*}
\begin{split}
\int_{\R^{n}} u^{\gamma} \, d \nu[u]
& \approx
\int_{\R^{n}} (\W_{1, p} \nu[u])^{\gamma} \, d \nu[u] \\
& \approx
\sum_{m = 1}^{M}
\int_{\R^{n}} (\W_{1, p}(u^{q_{m}} d \sigma^{(m)} ))^{\gamma} u^{q_{m}} \, d \sigma^{(m)}
+
\int_{\R^{n}} (\W_{1, p} \mu)^{\gamma} \, d \mu.
\end{split}
\end{equation*}
Hence, $u$ satisfies \eqref{generalized_energy}.
Positivity and minimality of $u$ follow from
\eqref{lower_bound_several_terms} and \eqref{upper_bound_several_terms}, respectively.
Thus, $u$ is a minimal $p$-superharmonic solution to \eqref{eq:several_terms}.
\end{proof}

We also prove the uniqueness of finite energy solutions
using a convexity argument as in \cite{MR3311903} and \cite{MR4048382}.

\begin{Cor}
Assume that \eqref{cond_q_several}-\eqref{cond_mu_several} hold with $\gamma = 1$.
Then there exists a unique finite energy weak solution $u \in \dot{W}_{0}^{1, p}(\R^{n})$
satisfying \eqref{eq:several_terms}.
\end{Cor}

\begin{proof}
Existence of a minimal solution follow from
Theorem \ref{thm:main-p-laplacian_several_terms}.
Testing \eqref{eq:def_of_uj@several_terms_localized} with $u_{j + 1}$
and using monotonicity of $\{ u_{j} \}_{j = 1}^{\infty}$,
we have
\[
\begin{split}
\int_{2B} |\nabla u_{j + 1}|^{p} \, dx
& =
\sum_{m = 1}^{M}
\int_{2B} u_{j + 1} u_{j}^{q_{m}} \, d \sigma^{(m)}
+
\int_{2B} u_{j + 1} \, d \mu
\\
& \leq
\sum_{m = 1}^{M}
\| u_{j + 1} \|_{L^{1 + q_{m}}(2B, d \sigma^{(m)})}^{1 + q_{m}}
+
\| u_{j + 1} \|_{L^{1}(2B, d \mu)}
\leq 
C,
\end{split}
\]
where $C$ is a constant depending only on $n$, $p$ and
\eqref{cond_q_several}-\eqref{cond_mu_several}.
Thus, the limit function $u$ belongs to $\dot{W}_{0}^{1, p}(\R^{n})$.

Let us prove uniqueness.
For simplicity, we put $\sigma^{(0)} = \mu$ and $q_{0} = 0$.
Let $u, v \in \dot{W}_{0}^{1, p}(\R^{n})$ be weak solutions to \eqref{eq:several_terms}.
Without loss of generality,
we may assume that $u$ and $v$ are quasicontinuous on $\R^{n}$
and that $v$ is minimal.
Hence, $u \geq v$ q.e. in $\R^{n}$.
Since each $\sigma^{(m)}$ is absolutely continuous with respect to the $p$-capacity,
$u \geq v$ $d \sigma^{(m)}$-a.e. for all $m = 0, 1, \dots, M$.

Testing the equations of $u$ with $u$,
we have
\begin{equation}\label{lem:uniqueness_several_04}
\int_{\R^{n}}
| \nabla u |^{p}
\, dx
=
\sum_{m = 0}^{M}
\int_{\R^{n}} u^{1 + q_{m}} \,d \sigma^{(m)}.
\end{equation}
By the same way,
\begin{equation}\label{lem:uniqueness_several_05}
\int_{\R^{n}}
| \nabla v |^{p}
\, dx
=
\sum_{m = 0}^{M}
\int_{\R^{n}} v^{1 + q_{m}} \,d \sigma^{(m)}.
\end{equation}
For $t \in [0, 1]$, we set
\[
\lambda_{t}(x)
:=
\left(
(1 - t) v^{p}(x) + t u^{p}(x)
\right)^{\frac{1}{p}}.
\]
Then
\begin{equation}\label{lem:uniqueness_several_07}
0 \leq v \leq \lambda_{t} \leq u
\quad \text{q.e. in $\R^{n}$.}
\end{equation}
Moreover, by the hidden convexity (see \cite[Proposition 2.6]{MR3273896}),
\begin{equation}\label{lem:uniqueness_several_06}
\int_{\R^{n}} |\nabla \lambda_{t} |^{p} \, dx
\leq
(1 - t) \int_{\R^{n}} |\nabla v |^{p} \, dx
+
t \int_{\R^{n}} |\nabla u |^{p} \, dx.
\end{equation}
Combining
\eqref{lem:uniqueness_several_04}, \eqref{lem:uniqueness_several_05}
and \eqref{lem:uniqueness_several_06},
we get
\begin{equation*}\label{lem:uniqueness_several_03}
\int_{\R^{n}} 
\frac{ |\nabla \lambda_{t} |^{p} - |\nabla \lambda_{0} |^{p} }{t} 
\, dx
\leq
\sum_{m = 0}^{M}
\left(
\int_{\R^{n}} u^{1 + q_{m}} \,d \sigma^{(m)}
-
\int_{\R^{n}} v^{1 + q_{m}} \,d \sigma^{(m)}
\right).
\end{equation*}
By the inequality
\[
|a|^{p} - |b|^{p}
\geq
p |b|^{p - 2} b \cdot (a - b)
\quad
(a, b \in \R^{n}),
\]
we have
\[
|\nabla \lambda_{t}|^{p} - |\nabla \lambda_{0}|^{p} 
\geq
p |\nabla \lambda_{0}|^{p - 2} \nabla  \lambda_{0} \cdot (\nabla \lambda_{t} - \nabla \lambda_{0}).
\]
Therefore,
\begin{equation}\label{lem:uniqueness_several_01}
\begin{split}
&
p \int_{\R^{n}} 
|\nabla v|^{p - 2} \nabla v \cdot \frac{ \nabla ( \lambda_{t} -  \lambda_{0} ) }{t} 
\, dx \\
& \leq
\sum_{m = 0}^{M}
\left(
\int_{\R^{n}} u^{1 + q_{m}} \,d \sigma^{(m)}
-
\int_{\R^{n}} v^{1 + q_{m}} \,d \sigma^{(m)}
\right).
\end{split}
\end{equation}
On the other hand,
using methods in \cite[Lemma 1.25]{MR2305115},
we can see that $\lambda_{t} \in \dot{W}_{0}^{1, p}(\R^{n})$
from \eqref{lem:uniqueness_several_07} and \eqref{lem:uniqueness_several_06}.
Testing the equation of $v$ with
$\lambda_{t} - \lambda_{0} \in \dot{W}_{0}^{1, p}(\R^{n})$, we get
\begin{equation}\label{lem:uniqueness_several_02}
\begin{split}
\int_{\R^{n}} 
|\nabla v|^{p - 2} \nabla v \cdot \nabla ( \lambda_{t} -  \lambda_{0} )
\, dx
=
\sum_{m = 0}^{M}
\int_{\R^{n}}  ( \lambda_{t} -  \lambda_{0} ) v^{q_{m}} \,d \sigma^{(m)}.
\end{split}
\end{equation}
Combining \eqref{lem:uniqueness_several_01} and \eqref{lem:uniqueness_several_02},
we obtain
\[
\sum_{m = 0}^{M}
p \int_{\R^{n}}  
\frac{ ( \lambda_{t} -  \lambda_{0} ) }{t} v^{q_{m}} \,d \sigma^{(m)}
\leq
\sum_{m = 0}^{M}
\left(
\int_{\R^{n}} u^{1 + q_{m}} \,d \sigma^{(m)}
-
\int_{\R^{n}} v^{1 + q_{m}} \,d \sigma^{(m)}
\right).
\]
Note that $\lambda_{t} - \lambda_{0} \geq 0$ $d \sigma^{(m)}$-a.e. 
for all $m$.
Therefore, by Fatou's lemma,
\[
\int_{\R^{n}} 
\left(
\frac{u^{p}}{v^{p - 1}} - v
\right)
v^{q_{m}} \,d \sigma^{(m)}
\leq
\liminf_{t \to 0}
p \int_{\R^{n}}  
\frac{ ( \lambda_{t} -  \lambda_{0} ) }{t} v^{q_{m}} \,d \sigma^{(m)}
\]
for each $m$.
Hence, passing to the limit $t \to 0$, we arrive at
\[
\sum_{m = 0}^{M}
\int_{\R^{n}} 
\left(
\frac{u^{p}}{v^{p - 1}} - v
\right)
v^{q_{m}} \,d \sigma^{(m)}
\leq
\sum_{m = 0}^{M}
\left(
\int_{\R^{n}} u^{1 + q_{m}} \,d \sigma^{(m)}
-
\int_{\R^{n}} v^{1 + q_{m}} \,d \sigma^{(m)}
\right),
\]
or equivalently
\[
\sum_{m = 0}^{M}
\int_{\R^{n}} 
\left(
\frac{ u^{p} v^{ q_{m} } }{v^{p - 1}} 
- 
u^{1 + q_{m}}
\right)
\,d \sigma^{(m)}
\leq 0.
\]
By minimality of $v$,
$v^{q_{m} - p + 1} \geq u^{q_{m} - p + 1}$ $d \sigma^{(m)}$-a.e. for all $m$,
and hence, each integral on the left-hand side is nonnegative.
Thus, $u = v$ $d \sigma^{(m)}$-a.e. for all $m$.

Testing the equations of $u$ and $v$ with $u - v \in \dot{W}_{0}^{1, p}(\R^{n})$, we get
\[
\int_{\R^{n}}
\left( |\nabla u|^{p - 2} \nabla u - |\nabla v|^{p - 2} \nabla v \right)
\cdot \nabla (u - v)
\, dx
= 0.
\]
This implies that $\nabla u = \nabla v$ a.e. in $\R^{n}$ (see, e.g., \cite[Lemma 5.6]{MR2305115}),
and therefore, $u = v$ in $\dot{W}_{0}^{1, p}(\R^{n})$.
\end{proof}

\section*{Acknowledgments}
The authors would like to thank the referee for his/her valuable comments 
which helped to improve the manuscript.
This work was supported by JSPS KAKENHI Grant Number JP18J00965 and JP17H01092.


\bibliographystyle{abbrv} 
\bibliography{reference}

\begin{thebibliography}{10}

\bibitem{MR1411441}
D.~R. Adams and L.~I. Hedberg.
\newblock {\em Function spaces and potential theory}, volume 314 of {\em
  Grundlehren der Mathematischen Wissenschaften [Fundamental Principles of
  Mathematical Sciences]}.
\newblock Springer-Verlag, Berlin, 1996.

\bibitem{MR1618334}
W.~Allegretto and Y.~X. Huang.
\newblock A {P}icone's identity for the {$p$}-{L}aplacian and applications.
\newblock {\em Nonlinear Anal.}, 32(7):819--830, 1998.

\bibitem{MR1354907}
P.~B\'{e}nilan, L.~Boccardo, T.~Gallou\"{e}t, R.~Gariepy, M.~Pierre, and J.~L.
  V\'{a}zquez.
\newblock An {$L^1$}-theory of existence and uniqueness of solutions of
  nonlinear elliptic equations.
\newblock {\em Ann. Scuola Norm. Sup. Pisa Cl. Sci. (4)}, 22(2):241--273, 1995.

\bibitem{MR1955596}
M.~F. Bidaut-V\'{e}ron.
\newblock Removable singularities and existence for a quasilinear equation with
  absorption or source term and measure data.
\newblock {\em Adv. Nonlinear Stud.}, 3(1):25--63, 2003.

\bibitem{MR1409661}
L.~Boccardo, T.~Gallou\"{e}t, and L.~Orsina.
\newblock Existence and uniqueness of entropy solutions for nonlinear elliptic
  equations with measure data.
\newblock {\em Ann. Inst. H. Poincar\'{e} Anal. Non Lin\'{e}aire},
  13(5):539--551, 1996.

\bibitem{MR1272564}
L.~Boccardo and L.~Orsina.
\newblock Sublinear equations in {$L^s$}.
\newblock {\em Houston J. Math.}, 20(1):99--114, 1994.

\bibitem{MR3273896}
L.~Brasco and G.~Franzina.
\newblock Convexity properties of {D}irichlet integrals and {P}icone-type
  inequalities.
\newblock {\em Kodai Math. J.}, 37(3):769--799, 2014.

\bibitem{MR1141779}
H.~Brezis and S.~Kamin.
\newblock Sublinear elliptic equations in {${\bf R}^n$}.
\newblock {\em Manuscripta Math.}, 74(1):87--106, 1992.

\bibitem{MR3567503}
D.~Cao and I.~Verbitsky.
\newblock Nonlinear elliptic equations and intrinsic potentials of {W}olff
  type.
\newblock {\em J. Funct. Anal.}, 272(1):112--165, 2017.

\bibitem{MR3311903}
D.~T. Cao and I.~E. Verbitsky.
\newblock Finite energy solutions of quasilinear elliptic equations with
  sub-natural growth terms.
\newblock {\em Calc. Var. Partial Differential Equations}, 52(3-4):529--546,
  2015.

\bibitem{MR3556326}
D.~T. Cao and I.~E. Verbitsky.
\newblock Pointwise estimates of {B}rezis-{K}amin type for solutions of
  sublinear elliptic equations.
\newblock {\em Nonlinear Anal.}, 146:1--19, 2016.

\bibitem{MR1734322}
C.~Cascante, J.~M. Ortega, and I.~E. Verbitsky.
\newblock Trace inequalities of {S}obolev type in the upper triangle case.
\newblock {\em Proc. London Math. Soc. (3)}, 80(2):391--414, 2000.

\bibitem{MR1760541}
G.~Dal~Maso, F.~Murat, L.~Orsina, and A.~Prignet.
\newblock Renormalized solutions of elliptic equations with general measure
  data.
\newblock {\em Ann. Scuola Norm. Sup. Pisa Cl. Sci. (4)}, 28(4):741--808, 1999.

\bibitem{MR2445437}
L.~Grafakos.
\newblock {\em Classical {F}ourier analysis}, volume 249 of {\em Graduate Texts
  in Mathematics}.
\newblock Springer, New York, second edition, 2008.

\bibitem{MR727526}
L.~I. Hedberg and T.~H. Wolff.
\newblock Thin sets in nonlinear potential theory.
\newblock {\em Ann. Inst. Fourier (Grenoble)}, 33(4):161--187, 1983.

\bibitem{MR2305115}
J.~Heinonen, T.~Kilpel\"{a}inen, and O.~Martio.
\newblock {\em Nonlinear potential theory of degenerate elliptic equations}.
\newblock Dover Publications, Inc., Mineola, NY, 2006.
\newblock Unabridged republication of the 1993 original.

\bibitem{MR2859927}
T.~Kilpel\"{a}inen, T.~Kuusi, and A.~Tuhola-Kujanp\"{a}\"{a}.
\newblock Superharmonic functions are locally renormalized solutions.
\newblock {\em Ann. Inst. H. Poincar\'{e} Anal. Non Lin\'{e}aire},
  28(6):775--795, 2011.

\bibitem{MR1205885}
T.~Kilpel\"{a}inen and J.~Mal\'{y}.
\newblock Degenerate elliptic equations with measure data and nonlinear
  potentials.
\newblock {\em Ann. Scuola Norm. Sup. Pisa Cl. Sci. (4)}, 19(4):591--613, 1992.

\bibitem{MR1264000}
T.~Kilpel{\"a}inen and J.~Mal{\'y}.
\newblock The {W}iener test and potential estimates for quasilinear elliptic
  equations.
\newblock {\em Acta Math.}, 172(1):137--161, 1994.

\bibitem{MR1402674}
T.~Kilpel\"{a}inen and X.~Xu.
\newblock On the uniqueness problem for quasilinear elliptic equations
  involving measures.
\newblock {\em Rev. Mat. Iberoamericana}, 12(2):461--475, 1996.

\bibitem{MR567696}
D.~Kinderlehrer and G.~Stampacchia.
\newblock {\em An introduction to variational inequalities and their
  applications}, volume~88 of {\em Pure and Applied Mathematics}.
\newblock Academic Press, Inc. [Harcourt Brace Jovanovich, Publishers], New
  York-London, 1980.

\bibitem{MR0409858}
V.~G. Maz$'$ja and V.~P. Havin.
\newblock A nonlinear potential theory.
\newblock {\em Uspehi Mat. Nauk}, 27(6):67--138, 1972.

\bibitem{MR3985926}
A.~Seesanea and I.~E. Verbitsky.
\newblock Solutions in {L}ebesgue spaces to nonlinear elliptic equations with
  subnatural growth terms.
\newblock {\em Algebra i Analiz}, 31(3):216--238, 2019.

\bibitem{MR3881877}
A.~Seesanea and I.~E. Verbitsky.
\newblock Solutions to sublinear elliptic equations with finite generalized
  energy.
\newblock {\em Calc. Var. Partial Differential Equations}, 58(1):Art. 6, 21,
  2019.

\bibitem{MR4048382}
A.~Seesanea and I.~E. Verbitsky.
\newblock Finite energy solutions to inhomogeneous nonlinear elliptic equations
  with sub-natural growth terms.
\newblock {\em Adv. Calc. Var.}, 13(1):53--74, 2020.

\bibitem{MR1890997}
N.~S. Trudinger and X.-J. Wang.
\newblock On the weak continuity of elliptic operators and applications to
  potential theory.
\newblock {\em Amer. J. Math.}, 124(2):369--410, 2002.

\bibitem{MR1747901}
I.~E. Verbitsky.
\newblock Nonlinear potentials and trace inequalities.
\newblock In {\em The {M}az$'$ya anniversary collection, {V}ol. 2 ({R}ostock,
  1998)}, volume 110 of {\em Oper. Theory Adv. Appl.}, pages 323--343.
  Birkh\"{a}user, Basel, 1999.

\bibitem{MR4030348}
I.~E. Verbitsky.
\newblock Quasilinear elliptic equations with sub-natural growth terms and
  nonlinear potential theory.
\newblock {\em Atti Accad. Naz. Lincei Rend. Lincei Mat. Appl.},
  30(4):733--758, 2019.

\bibitem{VERBITSKY2019111516}
I.~E. Verbitsky.
\newblock Wolff's inequality for intrinsic nonlinear potentials and quasilinear
  elliptic equations.
\newblock {\em Nonlinear Analysis}, 2020.
\newblock \url{https://doi.org/10.1016/j.na.2019.04.015}.

\end{thebibliography}


\end{document}